\documentclass[12pt]{amsart}

\usepackage{amsmath, amssymb, amscd, verbatim, xspace,amsthm}
\usepackage{latexsym, epsfig, color}
\usepackage[hidelinks]{hyperref}

\usepackage[OT2,T1]{fontenc}
\DeclareSymbolFont{cyrletters}{OT2}{wncyr}{m}{n}
\DeclareMathSymbol{\Sha}{\mathalpha}{cyrletters}{"58}

\newcommand{\ba}{\begin{align*}}
\newcommand{\ea}{\end{align*}}

\newcommand{\C}{\ensuremath{{\mathbb{C}}}}
\newcommand{\Z}{\ensuremath{{\mathbb{Z}}}\xspace}
\renewcommand{\P}{\ensuremath{{\mathbb{P}}}}

\newcommand{\Q}{\ensuremath{{\mathbb{Q}}}}
\newcommand{\R}{\ensuremath{{\mathbb{R}}}}
\newcommand{\F}{\ensuremath{{\mathbb{F}}}}
\newcommand{\G}{\ensuremath{{\mathbb{G}}}}

\newcommand{\E}{\ensuremath{{\mathbb{E}}}}

\newcommand{\ra}{\rightarrow}

\newcommand\Hom{\operatorname{Hom}}

\newcommand\Aut{\operatorname{Aut}}

\newcommand\im{\operatorname{im}}

\newcommand\Gal{\operatorname{Gal}}
\newcommand\Nm{\operatorname{Nm}}

\newcommand\Sur{\operatorname{Sur}}

\newcommand\Ind{\operatorname{Ind}}

\newcommand\tensor{\otimes}
\newcommand\isom{\simeq}
\newcommand\sub{\subset}
\newcommand\tesnor{\otimes}

\newcommand\Disc{\operatorname{Disc}}

\newcommand\tr{\operatorname{tr}}

\renewcommand\O{\mathcal{O}}

\newcommand\bq{\begin{equation}}
\newcommand\eq{\end{equation}}

\numberwithin{equation}{section}

\newtheorem{proposition}[equation]{Proposition}
\newtheorem{theorem}[equation]{Theorem}
\newtheorem{corollary}[equation]{Corollary}
\newtheorem{example}[equation]{Example}

\newtheorem{lemma}[equation]{Lemma}

\newtheorem{conjecture}[equation]{Conjecture}

\theoremstyle{remark}
\newtheorem{remark}[equation]{Remark}

\usepackage{fullpage}
\usepackage{pdfsync}
\usepackage{graphicx,amsfonts,enumerate,mathrsfs,tikz,tikz-cd,mathtools,appendix}

\newenvironment{definition}{\vspace{2 ex}{\noindent{\bf Definition. }}}{\vspace{2 ex}}

\newtheorem{nts}{Note to self}

\newcommand{\melanie}[1]{{\color{blue} \sf $\clubsuit\clubsuit\clubsuit$ Melanie: [#1]}}

\newcommand\cut[1]{}

\newcommand\further[1]{}

\newcommand\Cl{\operatorname{Cl}}
\newcommand\fO{\mathfrak{O}}
\newcommand\fo{\mathfrak{o}}
\newcommand\so{{\scriptstyle\mathcal{O}}}
\newcommand\fp{\mathfrak{p}}
\newcommand\fP{\mathfrak{P}}
\newcommand\uu{\underline{u}}
\newcommand\uv{\underline{v}}
\newcommand\us{\underline{s}}
\newcommand\ui{\underline{\infty}}

\title{Moments and interpretations of the {C}ohen-{L}enstra-{M}artinet heuristics}

\author{Weitong Wang}
\address{Department of Mathematics\\
University of California, Berkeley\\
970 Evans Hall \# 3840\\
Berkeley, CA 94720-3840 USA}  
\email{
tom\_wang@berkeley.edu}

\author{Melanie Matchett Wood}
\address{Department of Mathematics\\
University of California, Berkeley\\
970 Evans Hall \# 3840\\
Berkeley, CA 94720-3840 USA}  
\email{mmwood@berkeley.edu}

\begin{document}
\begin{abstract}
The goal of this paper is to prove theorems that elucidate the Cohen-Lenstra-Martinet conjectures for the distributions of class groups of number fields, and  further the understanding of their implications.
We start by giving a  simpler statement of the  conjectures.  We show that the probabilities that arise are inversely proportional  to the number of automorphisms of structures slightly larger than the class groups.  We find the moments of the Cohen-Lenstra-Martinet distributions and prove that the distributions are determined by their moments.  In order to apply these conjectures to class groups of non-Galois fields, we prove a new theorem on the capitulation kernel (of ideal classes that become trivial in a larger field) to relate the class groups of non-Galois fields to the class groups of Galois fields.  We then construct an integral model of the Hecke algebra of a finite group, show that it acts naturally on class groups of non-Galois fields, and prove that the Cohen-Lenstra-Martinet conjectures predict a distribution  for class groups of non-Galois fields that involves the inverse of the number of automorphisms of the class group as a Hecke-module.  
\end{abstract}

\maketitle

\section{Introduction}

In this paper we prove several results to help elucidate the Cohen-Lenstra-Martinet conjectures \cite{CL84, CM90} for the distributions of class groups of number fields, and to further the understanding of their implications.
In Section~\ref{S:C1}, we explain the statement of the conjectures in the framework of probability theory. In Section~\ref{S:u}, we prove a result about  the terms appearing in the Cohen-Lenstra-Martinet probabilities, so that we have the following statement. (See Conjecture~\ref{C:Galois} and Theorem~\ref{thm:u in Cohen-Martinet} for precise statements.)

\begin{theorem}\label{T:CM}
For every finite group $\Gamma$ and subgroup $\Gamma_\infty$, among Galois number fields $K$ with isomorphism $\Gal(K/\Q)\isom \Gamma$ (i.e. \emph{$\Gamma$-fields}) and decomposition group $\Gamma_\infty$ at $\infty$, the Cohen-Lenstra-Martinet conjectures predict that 
$\Cl_K\tensor_{\Z} \Z[|\Gamma|^{-1}] $ is isomorphic to
$$
H \text{ with probability  } \frac{c}{|H^{\Gamma_\infty}||\Aut_\Gamma(H)|},
$$
where $\Cl_K$ is the class group of $K$,  and $c$ is a constant, and $H$ is any finite $\Z[|\Gamma|^{-1},\Gamma]$-module with $H^\Gamma=1$.
\end{theorem}

The original philosophy of the Cohen-Lenstra-Martinet conjectures, going back to Cohen and Lenstra \cite{CL84}, is that objects should appear with frequency inversely proportional to their number of automorphisms.  So we naturally ask why there is an
$|H^{\Gamma_\infty}|$ term in the above predictions. In Section~\ref{S:getaut}, we slightly enlargen the class group to the Galois group over $\Q$ of the Hilbert class field of $K$, with the data of a decomposition group at $\infty$.  We then see that for this slightly larger structure, which we call a \emph{class triple} and is determined by the class group and decomposition group, the number of automorphisms of this structure is exactly $|H^{\Gamma_\infty}||\Aut_\Gamma(H)|$, explaining the probabilities above.  Bartel and Lenstra \cite{Bartel2018} have given a different approach to this question by giving conjectures about the distribution of Arakelov class groups based on those groups appearing with frequency inversely proportional to their number of automorphisms (which takes some work to make precise, see \cite{Bartel2015}).  Their predicted distribution on Arakelov class groups then pushes forward to the Cohen-Lenstra-Martinet distribution, over any base number field.

In Section~\ref{S:moments}, we determine the moments, which are important averages of the Cohen-Lenstra-Martinet distributions on finite abelian $\Gamma$-modules. 

\begin{theorem}[Moments]\label{T:Mom} For every finite group $\Gamma$ and subgroup $\Gamma_\infty$, if $X$ is a random $\Z[|\Gamma|^{-1},\Gamma]$-module with the Cohen-Lenstra-Martinet distribution for $\Gamma$-fields with decomposition group $\Gamma_\infty$ at $\infty$, then 
for every finite $\Z[|\Gamma|^{-1},\Gamma]$-module $H$ with $H^{\Gamma}=1$,
we have the \emph{$H$-moment} of $X$ is
$$
\E(|\Sur_\Gamma(X,H)|)=|H^{\Gamma_\infty}|^{-1}.
$$
\end{theorem}
Here $\Sur_\Gamma(X,H)$ denotes the surjective $\Gamma$-module homomorphisms from $X$ to $H$.  See 
Theorem~\ref{thm:u in Cohen-Martinet} and Theorem~\ref{thm: formula for moments} for precise statements.
These moments are the most important averages of the Cohen-Lenstra-Martinet distributions. (See \cite[Section 3.3]{Clancy2015} on why they are called moments.)  The only non-trivial predicted averages of the Cohen-Lenstra-Martinet conjectures that have been proven are
  the $\Z/3\Z$-moment of the class groups of quadratic fields due to Davenport and Heilbronn  
 \cite{Davenport1971} 
(and Datskovsky and Wright \cite{Datskovsky1988} for quadratic extensions of general global fields) and the $\Z/2\Z$-moment of the class groups of cubic fields due to Bhargava \cite{Bhargava2005}.
(There is also more known on the $2$-Sylow subgroup of the class groups of quadratic fields; see  
\cite{Fouvry2006a,Smith2017}.)  When working over $\F_q(t)$ instead of $\Q$, there are also results on the $H$-moments of class groups, including of Ellenberg, Venkatesh, and Westerland \cite{Ellenberg2016} and the second author \cite{Wood2018} for quadratic extensions, and of Liu, the second author, and Zureick-Brown \cite{LWZB} for $\Gamma$-extensions, showing that as $q\ra\infty$ the moments match those in Theorem~\ref{T:Mom}.  The paper
\cite{Pierce2019} of Pierce, Turnage-Butterbaugh, and the second author explains how the Cohen-Lenstra-Martinet conjectures for the moments of class groups are related to other important conjectures in number theory, including the $\ell$-torsion conjecture for class groups, the discriminant multiplicity conjecture, generalized Malle's conjecture, and the count of elliptic curves with fixed conductor.  
 So given the relative accessibility and the centrality of these moments, Theorem~\ref{T:Mom} is useful because it tells us what moments the Cohen-Lenstra-Martinet conjectures predict.

 Moreover, we show that moments determine the Cohen-Lenstra-Martinet distributions uniquely, which is
particularly of interest because the moments are the statistics of class groups about which we seem most likely to be able to prove something. 
 
\begin{theorem}[Moments determine distribution]\label{T:InttoMomDet}
For every finite group $\Gamma$ and subgroup $\Gamma_\infty$, if $X$ is a random $\Z[|\Gamma|^{-1},\Gamma]$-module 
such that for every finite $\Z[|\Gamma|^{-1},\Gamma]$-module $H$ with $H^{\Gamma}=1$,
we have
$$
\E(|\Sur_\Gamma(X,H)|)=|H^{\Gamma_\infty}|^{-1}.
$$
then $X$ has the Cohen-Lenstra-Martinet distribution for $\Gamma$-fields with decomposition group $\Gamma_\infty$ at $\infty$.
\end{theorem} 

See Theorems~\ref{thm: moments determine distribution finite case} and \ref{T:infmom} for precise statements.   When we restrict to  groups whose orders are only divisible by a finite set of primes, we also prove that a sequence of random variables with these moments in the limit must have the Cohen-Lenstra-Martinet distribution as its limit distribution.  Theorem~\ref{T:InttoMomDet} is part of a long line of work showing results in the same spirit for other categories of groups, including work of Heath-Brown \cite[Lemma 17]{Heath-Brown1994}
for elementary abelian $p$-groups, 
Ellenberg, Venkatesh, and Westerland \cite[Section 8]{Ellenberg2016}
for finite abelian $p$-groups, 
the second author for finite abelian groups \cite[Section 8]{MW17}, and Boston and the second author \cite[Theorem 1.4]{Boston2017} for pro-$p$ groups with a $\Z/2\Z$ action.
See  \cite{Delaunay2014a,Fouvry2006,Fulman2019,Wood2018} for other examples.

Next, we consider the implications of the Cohen-Martinet conjecture for class groups of non-Galois fields.  While these conjectures do not directly make claims about class groups of non-Galois fields, when the class groups of non-Galois fields can be given as a function of the class groups of Galois fields, then the Cohen-Martinet conjectures make a prediction for their average.  For example, let $\Gamma$ be a finite group and $\Gamma'$ a subgroup of $\Gamma$.
When $L$ is a $\Gamma$-field and $K$ is the fixed field $L^{\Gamma'}$, then, localizing away from primes dividing $|\Gamma|$, we have 
$\Cl_K \tesnor_\Z \Z[|\Gamma|^{-1}]=(\Cl_L^{\Gamma'})\tesnor_\Z \Z[|\Gamma|^{-1}]$ (where the $\Gamma'$ exponent denotes taking the fixed part).  However, there is also the possibility of using the Cohen-Martinet conjectures, for some primes $p\mid |\Gamma|$,
to predict distributions of  $p$-Sylow subgroups $\Cl_{K,p}$ of $\Cl_K$.
In order to realize this possibility, we prove a new result relating class groups of non-Galois fields to class groups of Galois fields, in particular at primes dividing the order of the Galois group.

\begin{theorem}[Determination of class groups of non-Galois fields from Galois]\label{T:getCl}
Let $L/K$ be an extension of number fields such that  $L/\Q$ is Galois with Galois group $\Gamma$ and let $\Gamma'=\Gal(L/K)$.
Let $e_{\Gamma/\Gamma'}$ be the central idempotent of $\Q[\Gamma]$ for the augmentation character for $\Gamma$ acting on $\Gamma'$ cosets,
and $p$ a prime not dividing the denominator of $e_{\Gamma/\Gamma'}$ and such that $e_{\Gamma/\Gamma'}\Z_{(p)}[\Gamma]$ is a maximal order. 
Then we have an isomorphism
 \[\Cl_{K,p}\stackrel{\sim}{\longrightarrow}\big(e_{\Gamma/\Gamma'}\Cl_{L,p}\big)^{\Gamma'},\]
 where the subscript $p$ denotes taking the Sylow $p$-subgroup.
\end{theorem}

See Theorem~\ref{thm: Gamma'-invariant} for a precise statement (for relative class groups over an arbitrary base number field).  In particular, we note the restriction on $p$ is exactly the condition  on $p$  for the Cohen-Martinet conjectures to say something about the distribution of 
$e_{\Gamma/\Gamma'}\Cl_{L,p}$.  So Theorem~\ref{T:getCl} allows us to fully determine the implications of the Cohen-Martinet conjectures
for the class groups of non-Galois fields.

Moreover, for $p,K,L$ as in Theorem~\ref{T:getCl}, we have the immediate corollary that the order of the kernel of the capitulation map $\Cl_K \ra \Cl_L$ is not divisible by $p$.  The capitulation kernel is very long-studied, but its structure is not well-known.  Hilbert's Theorem 94 \cite{Hilbert1998} proves that when $L/K$ is finite, cyclic, and unramified, then the degree $[L:K]$ divides the order of the capitulation kernel. Hilbert then conjectured the Principal Ideal Theorem of class field
theory, eventually proved by  Artin and F\"{u}rtwangler, that every ideal class in $K$ capitulates in the Hilbert class field.  Suzuki \cite{Suzuki1991a} and Gruenberg and Weiss \cite{Gruenberg2000} proved further generalizations showing that the capitulation kernel for unramified abelian extensions is large. 
Our theorem above is in the other direction, proving in some cases there is no $p$-part of the capitulation kernel.

Theorem~\ref{T:getCl} implies that the Cohen-Martinet conjectures in principle give a prediction for the distribution of class groups of fields $K$ as above, but the predicted distribution for a finite abelian $p$-group $H$ is then the sum over  $e_{\Gamma/\Gamma'}\Z_{(p)}[\Gamma]$-modules $G$ such that 
$G^{\Gamma'}\isom H$ (as groups) of the probability for $G$ in the Galois predictions (see Equation~\eqref{eqn: probability of non-Galois fields in section 6}).  This prediction does not  have the appearance of objects appearing with frequency inversely proportional to their number of automorphisms.  
However, in Section~\ref{S:Re}, we prove new theorems to give such a perspective on these probabilities.

Of course when $L/\Q$ is Galois, we have that $\Gal(L/\Q)$ acts on $\Cl_L$.  However, when $K/\Q$ has no automorphisms, one might at first guess that $\Cl_K$ has no particular structure other than that of a finite abelian group.  We prove, however, that there is always a natural action of a certain ring $\fo$ on $\Cl_K$ (depending on the Galois groups of the Galois closure over $\Q$ and $K$).
  Given a representation $V$ of  finite group $\Gamma$ over $\Q$ and a subgroup $\Gamma' $ of $\Gamma$, the Hecke algebra $\Q[\Gamma'\backslash\Gamma/\Gamma']$ naturally acts on $V^{\Gamma'}$.  We construct an integral model $\fo$ of the Hecke algebra
so that the class group $\Cl_{K,p}$ (for $K,p,\Gamma,\Gamma'$ as in Theorem~\ref{T:getCl}) is naturally an $\fo$-module (see Lemma~\ref{lemma: Gamma'-invariant part is an o-module}) and prove that our constructed $\fo$ is a maximal order (Corollary~\ref{cor: o is a maximal order}).  This definition of $\fo$  is  particularly delicate at the primes $p\mid |\Gamma'|$, but the proofs require similar work at all $p$.
Note that $\fo$ can be bigger than $\Z$ even when the field $K$ has no automorphisms; see Example~\ref{Ex:A5} on degree $10$ fields with Galois closure with group $A_5$ and Proposition~\ref{Prop: criterion} in which we prove $\fo$ is trivial if and only if the augmentation character for $\Gamma$ acting on $\Gamma'$ cosets is absolutely irreducible. 

Moreover, Theorem~\ref{T:getCl} and the results in Section \ref{S:Re} show that the $p$-Sylow subgroup of the $\Gamma$-module $\Cl_{L,p}$ of a Galois field $L$ containing $K$ determines the $\fo$-module structure of $\Cl_{K,p}$.  So we have shown that the Cohen-Martinet conjectures imply some prediction for the distribution of the $\fo$-modules $\Cl_{K,p}$, and we further prove a simple expression for the prediction in terms of $|\Aut_\fo(H)|^{-1}$ by way of the following result.  

\begin{theorem}[Cohen-Martinet predicts $|\Aut_\fo(H)|^{-1}$ for non-Galois fields]\label{T:nonGalprobs}
Given a finite group $\Gamma$ and subgroup $\Gamma'$, 
for every prime $p$ satisfying the condition of Theorem~\ref{T:getCl}, and every $p$-group $\fo$-module $H$,
there is a unique finite $e_{\Gamma/\Gamma'}\Z_{(p)}[\Gamma]$-module $G$ such that
$
G^{\Gamma'}\cong H
$ as $\fo$-modules.  We also have
$$
\Aut_{e_{\Gamma/\Gamma'}\Z_{(p)}}(G)\isom \Aut_{\fo}(H).
$$
\end{theorem}
See Theorem~\ref{thm: punchline} for a related statement precisely on the implications of the Cohen-Martinet conjecture.
The key result we prove that allows us to prove Theorem~\ref{T:nonGalprobs} is Theorem~\ref{theorem: equivalence of categories}, which gives a Morita equivalence between the categories of $e_{\Gamma/\Gamma'}\Z_{(p)}[\Gamma]$-modules and $\fo$-modules.  This is the fundamental algebraic property of our integral model $\fo$ of the Hecke algebra.

Note that Theorem~\ref{T:getCl} does not require $L$ to be the Galois closure of $K$.  So actually, the Cohen-Lenstra-Martinet heuristics give infinitely many different predictions for the distribution of non-Galois (or Galois) class groups, by taking fixed fields of larger and larger fields.  In Section~\ref{S:ind}, we prove that all of the predicted distributions agree, which is an important internal consistency check on the conjectures.

Theorems~\ref{T:CM}, \ref{T:Mom}, \ref{T:InttoMomDet}, and \ref{T:nonGalprobs} are theorems in the theory of finite $\Gamma$-modules, including in the probability theory of random finite $\Gamma$-modules.  Even though we have proven them to specifically elucidate conjectures about class groups, we expect them, especially Theorems \ref{T:Mom} and \ref{T:InttoMomDet} to have applications in other contexts. 
Distributions related to the Cohen-Lenstra distribution have arisen for predicting the distribution of Tate-Shafarevich groups of elliptic curves \cite{Delaunay2001,Bhargava2015b}, and so in order to generalize the predictions of 
\cite{Park2016} on the asymptotics of elliptic curves of a given rank  over $\Q$ to other base global fields, one will need to use an analog of the Cohen-Martinet distributions.
Also, beyond number theory, the Cohen-Lenstra distributions on finite abelian groups, and related distributions, have many interesting connections in algebraic combinatorics; see the recent work of Fulman and Kaplan \cite{Fulman2019} and also \cite{Chinta2017,Clancy2015, Clancy2015a,Fulman1999,Fulman2014,Fulman2016, Garton2016, Lengler2008, Lengler2010, Nguyen2016,Stanley2016, Wang2017}.  Further, the theorems that moments determine the distribution have been used for determining  distributions arising in the theory of random graphs, such as the sandpile groups of Erd\"{o}s-R\'{e}nyi and random regular graphs \cite{Koplewitz2017a, Meszaros2018,MW17}.  These theorems on the moments have also been used to show that certain random matrices have cokernels in the Cohen-Lenstra distribution \cite{Nguyen2018a, Nguyen2018b, Wood2015a}, and as an application determine the probability that a random $0/1$ rectangular matrix gives a surjective map to $\Z^n$.
The Cohen-Lenstra and related distributions have also arisen in questions about random topological spaces \cite{Dunfield2006,Kahle2017}.  The more general Cohen-Lenstra-Martinet distributions may be relevant in many of these contexts.

\section{Notation}\label{S:Notations}
Throughout the whole paper, $\Gamma$ is always a finite group and $S$ is always a set of (possibly infinitely many) rational primes.

\begin{definition} 
Let $K$ be a number field and $K_0/\Q$ be a subextension of $K$. 
We write $\Cl_K$ for the class group of $K$.
Then we define the \emph{relative class group} $\Cl_{K/K_0}$  to be the subgroup of $\Cl_K$ consisting of ideal classes $\alpha$ with trivial norm $\operatorname{Nm}_{K/K_0}\alpha$ in $\Cl_{K_0}$. Also, let $I_K$ be the group of fractional ideals and $P_K$ the group of principal fractional ideals of $K$.
\end{definition}

\begin{definition}
For a field $K_0$, by a \emph{$\Gamma$-extension of $K_0$}, we mean an isomorphism class of pairs $(K,\tau)$, where $K$ is a Galois extension of $K_0$, and $\tau: \Gal(K/K_0)\isom \Gamma$ is an isomorphism.  An isomorphism of pairs $(K,\tau)$, $(K',\tau')$ is an isomorphism $\alpha: K\ra K'$ such that the map $m_\alpha: \Gal(K/K_0) \ra \Gal(K'/K_0)$ sending $\phi$ to $ \alpha \circ \phi  \circ \alpha^{-1}$ satisfies $\tau' \circ m_\alpha=\tau$.  We sometimes leave the $\tau$ implicit, but this is always what we mean by a $\Gamma$-extension.  We also call $\Gamma$-extensions of $\Q$ \emph{$\Gamma$-fields.}

\end{definition}

\begin{definition}
Define $\Z_S$ to be the localization of $\Z$ by the subset of non-zero integers not divisible by any primes in $S$, so the maximal ideals of $\Z_S$ are given by the primes in $S$.
For any finite abelian group $G$, define its $S$ part $G^S$ as the subgroup generated by all $p$-Sylow subgroups with $p\in S$. 
(Note that our definition for $S$-part of $G$ is the opposite of $G^S$ in \cite{CM90}.)  We will also use the usual notation  $\Z_{(p)}$ for $\Z_S$ when $S=\{p\}$.
\end{definition}

\begin{definition}
If $f$ is a measurable function on a probability space, 
we let $\P$ denote the probability measure 
and  $ \mathbb{E}(f)$
denote the expected value of $f$.
In this paper,  our probability spaces will always be discrete and countable and
 \[\mathbb{E}(f)=\sum_{i=1}^\infty f(G_i)\mathbb{P}(G_i).\]
\end{definition}

Throughout the paper, we often have a ring $R$, a central idempotent $e$ of $R$, and then consider the ring $eR$.  The reader is warned that $eR$ is \emph{not} a subring of $R$ in the usual sense, as $R$ and $eR$ do not share an identity.  One could consider $eR$ as notation for the quotient $R/(1-e)R$.

\section{Explanation of the Cohen-Lenstra-Martinet Heuristics in the Galois case}\label{S:C1}

The goal of this section is to state Cohen, Lenstra, and Martinet's conjectures on the distribution of relative class groups of Galois extensions. This requires introducing many pieces of notation.

\subsection{Notations for semisimple  \texorpdfstring{$\Q$}{Q}-algebras}\label{S:ssNot}

Let $A$ be a finite dimensional semisimple $\Q$-algebra; we denote by $\{e_i\}_{1\leq i \leq m}$ its irreducible \emph{central} idempotents, and $A_i = e_iA$ its simple factors.
The algebra $A$ is thus identified with a product $\prod_{i=1}^m A_i$, 
where each algebra $A_i$ is isomorphic to an algebra of matrices $M_{l_i}(D_i),$ 
where $D_i$ is a division algebra of finite rank over $\Q$ of which the center is a number field $K_i$.  
We let $h_i^2=\dim_{K_i} A_i.$
Let $\fO$ be a maximal order in $A$ and $G$ a finite $\fO$-module. For any $\uu\in\Q^m$, we define
 \[|G|^{\uu}:=\prod_{i=1}^m|e_iG|^{u_i}.\]
 (See \cite[\S10]{R03maximal} for basic results on semisimple $\Q$-algebras and maximal orders.)

\subsection{Notations for the Heuristics}\label{S: notations for the Heuristics}

In the rest of this section, we let $A=\Q[\Gamma]$, and continue with the notation above. 
In particular, we let
 \[e_1=\frac{1}{|\Gamma|}\sum_{\sigma\in\Gamma}\sigma.\]
Each $e_i$ corresponds to a distinct irreducible $\Q$-representation of $\Gamma$ with character $\chi_i$.
We choose a fixed absolutely irreducible character $\varphi_i$ contained in $\chi_i$.

Now let $K_0$ be a number field, and $K/K_0$ a Galois extension with Galois group $\Gamma$. 
If $v$ is an infinite place $v$ of $K_0$, then let $\Gamma_v$ be the decomposition group at $v$.
We also define
 \[\chi_K=-1+\sum_{v\mid\infty}\operatorname{Ind}_{\Gamma_v}^\Gamma 1_{\Gamma_v},\]
which is a character of $\Gamma$ associated to $K/K_0$.

\begin{definition} We define the rank of $K|K_0$ to be an $m-1$-tuple in $\Q^{m-1}$ given by the formula
\begin{equation}\label{E:defrank}
\uu=(u_2,\dots,u_m),\,u_i=\frac{1}{h_i}\langle\chi_K,\varphi_i\rangle\,\quad\quad\forall i=2,\dots,m.
\end{equation}
\end{definition}
\begin{remark}

 For the original definition of rank of $K$, see \cite[Definition 6.4]{CM90}. These two definitions are equivalent by \cite[Theorem 6.7]{CM90}.

\end{remark}

Let $S$ be a finite set of primes.  We define the following random module to model the class groups $\Cl_{K}^S$, which are naturally $(1-e_1)\Z_S[\Gamma]$-modules.

\begin{definition} If $p\in S$ implies that $p\nmid|\Gamma|$, then 
for $\uu=(u_2,\dots,u_m)\in\Q^{m-1}$, we  
define a random variable $X=X((1-e_1)\Q[\Gamma],\uu,(1-e_1)\Z_S[\Gamma])$ to be a random $(1-e_1)\Z_S[\Gamma]$-module such that for all finite $(1-e_1)\Z_S[\Gamma]$-modules $G_1,G_2$, we have
 \[\frac{\mathbb{P}(X\cong G_1)}{\mathbb{P}(X\cong G_2)}=\frac{|G_2|^{\uu}|\Aut_\Gamma(G_2)|}{|G_1|^{\uu}|\Aut_\Gamma(G_1)|}\]
 (where, of course, we order the irreducible central idempotents of $(1-e_1)\Q[\Gamma]$ by the order
 in $\Q[\Gamma]$).
\end{definition}

\begin{remark}\label{remark:class groups have trivial invariant part}
\cut{\begin{enumerate}[(i)]

\item 

Note that by our assumption that primes in $S$ are relatively prime to $|\Gamma|$, we know that $e_i\in\Z_S[\Gamma]$ for every $i=1,\ldots,m$. This means that the ring $(1-e_1)\Z_S[\Gamma]$ is a component of the group ring $\Z_S[\Gamma]$. And for any $(1-e_1)\Z_S[\Gamma]$-module $G$, it is naturally a $\Z_S[\Gamma]$-module via the projection $\Z_S[\Gamma]\to(1-e_1)\Z_S[\Gamma]$ given by $x\mapsto(1-e_1)x$. In particular,
 \[G^\Gamma=e_1G=e_1(1-e_1)\cdot G=0\cdot G=0,\]
i.e., a $(1-e_1)\Z_S[\Gamma]$-module is just a $\Z_S[\Gamma]$-module with trivial $\Gamma$-invariant part.

\item  The $S$-part of the relative class group has trivial $\Gamma$-invariant part, because for any ideal class $x\in\Cl^S_{K|K_0}$,
 \[e_1\cdot x=\frac{1}{|\Gamma|}\Nm_{K|K_0}x=0\in\Cl^S_{K|K_0}.\]
This is the reason why we only consider the category of finite $(1-e_1)\Z_S[\Gamma]$-modules.
\item} 
It follows from \cite[Theorem 3.6]{CM90} (with their $\underline{u}$ as $\underline{\infty}$ and their $\underline{s}$ as our $\underline{u}$) that
this definition is well-defined, i.e., the series
 \[\sum_G\frac{1}{|G|^{\uu}|\Aut_\Gamma(G)|},\]
is convergent, where $G$ runs through all isomorphism classes of finite $(1-e_1)\Z_S[\Gamma]$-modules. 
Even when $|S|=\infty$, the series is still convergent as long as 
$u_i>0$ for all $i=1,\ldots,m$.
So the above definition can be extended to the case  $|S|=\infty$ as long as all the $u_i$'s are positive.
\cut{
\end{enumerate}}
\end{remark}

\subsection{Statement of the Conjecture}

The conjecture of Cohen-Martinet \cite[Hypothesis 6.6]{CM90} says the following.

\begin{conjecture}[Cohen and Martinet \cite{CM90}]\label{C:Galois} Let $S$ be a finite set of prime numbers such that the primes in $S$ are relatively prime to $|\Gamma|$, and $\uu\in \Q^{m-1}$, and $X=X((1-e_1)\Q[\Gamma],\uu,(1-e_1)\Z_S[\Gamma])$  the random module defined above. Then, for every ``reasonable'' non-negative function $f$ defined on the set of isomorphism classes of finite $(1-e_1)\Z_S[\Gamma]$-modules, we have
\[
\lim_{x\to\infty}\frac{\sum_{|\Disc K|\leq x}f((1-e_1)\Cl^S_{K})}{\sum_{|\Disc K|\leq x}1}=\E(f(X)),
\]
where the sum is over all $\Gamma$-extensions $K|K_0$ and the rank of $K|K_0$ is $\uu$
  (and no conjecture is made if the sums are empty).
\end{conjecture}

The cases when \(K_0=\Q\) and either $\Gamma$ is abelian and $K$ is totally real, or $|\Gamma|=2$, are the earlier conjecture of Cohen-Lenstra \cite[Fundamental Assumptions 8.1]{CL84}.

\begin{remark}
In \cite{CM90}, a quantity $M_{\uu}^S(f)$ appears in place of $\E(f(X))$.
The identity $M_{\uu}^S(f)=\mathbb{E}\big(f(X)\big)$ is proved in Proposition \ref{Prop: M(f)=E(f(X))}. 
Also the \(S\)-part of the relative class group \(\Cl^S_{K|K_0}\) appears in place of \((1-e_1)\Cl^S_K\). 
In Lemma \ref{lemma: e-relative class group=e-class group}, we show that these are actually the same.
Note that \(e_1\Cl^S_K=\Cl^S_{K_0}\). Therefore we only consider the \((1-e_1)\)-part as a random object.
\end{remark}

Cohen and Martinet actually make further conjectures for some primes dividing $|\Gamma|$ and for infinite $S$.  We will give the conjecture for $p\mid |\Gamma|$ in Conjecture~\ref{C:CMfull}.
Given the example of \cite[Theorem 1.1]{Bartel2018} of Bartel and Lenstra, it is probably best to keep the conjecture to finite sets $S$.  The ordering of the fields needs to be changed in the conjecture, 
given the example of \cite[Theorem 1.2]{Bartel2018} of Bartel and Lenstra, who suggest ordering fields by the radical of their discriminant based on work on the second author \cite{Wood2010} that shows this ordering has nice statistical properties for abelian Galois groups.  Malle's work \cite{Malle2008,Malle2010} suggests that we should also require that $S$ does not contain any primes dividing the order of the roots of unity of $K_0$.  
The function field results in \cite{LWZB} suggest that these are all the corrections that need to be made.
See \cite[Section 5.6]{Bhargava2015b} and \cite[Section 7]{Bartel2018} for some discussion of what a ``reasonable'' function might be.

\section{The \texorpdfstring{$|G|^{\uu}$}{|G|u} in Cohen-Martinet}\label{S:u}

In this section, we will find a simpler expression for the $|G|^{\uu}$ term that appears in the conjecture of Cohen and Martinet. We continue the notation from Section~\ref{S:C1}.

\begin{theorem}\label{thm:u in Cohen-Martinet}
Let $K|K_0$ be a $\Gamma$-extension of number fields. For each infinite prime $v$ of $K_0$, let $\Gamma_v$ be a decomposition group 
at $v$. We assume that the set $S$ only contains primes not dividing $|\Gamma|$. 
If $H$ is a finite $(1-e_1)\Z_S[\Gamma]$-module,
then
 \[|H|^{\uu}=\prod_{v|\infty}|H^{\Gamma_v}|,\]
where $v$ runs over all infinite primes of $K_0$.
\end{theorem}

\begin{proof}By the definition of $|H|^{\uu}$, the theorem reduces to the case of a
$\Z_S[\Gamma]$-module $H$ such that $H=e_iH$ for some $i>1$.
Let $e\neq e_1$ be a central irreducible idempotent of $\Q[\Gamma]$ associated to the $\Q$-irreducible character $\chi$ and rank $u$, and let $H$ be a finite $e\Z_S[\Gamma]$-module. We first show the following identity
 \[\lvert H^{\Gamma_v}\rvert=\lvert H\rvert^{\frac{\langle\chi,a_{\Gamma/\Gamma_v}\rangle}{\langle\chi,a_\Gamma\rangle}}\]
for each infinite place $v$ of $K_0$, where  for a 
subgroup $\Delta\subseteq\Gamma$ we define
$a_{\Gamma/\Delta}:=-1+\Ind_{\Delta}^\Gamma 1_{\Delta}$ to be the augmentation character of $\Delta$  and $a_\Gamma:=a_{\Gamma/1}$. By \cite[Theorem 7.3]{CM90}, for each $v$, there exists some abelian group $G_v$ such that, as abelian groups, we have
 \[H=eH\cong G_v^{\langle\chi,a_\Gamma\rangle}\quad\text{and}\quad H^{\Gamma_v}=(eH)^{\Gamma_v}\cong G_v^{\langle\chi,a_{\Gamma/\Gamma_v}\rangle},\]
hence the identity.

Note that $\chi_K=-1+\sum_{v|\infty}(a_{\Gamma/\Gamma_v}+1)$, and that $\langle\chi,1\rangle=0$. We then know that
 \[\prod_{v|\infty}|H^{\Gamma_v}|=\prod_{v|\infty}\lvert H\rvert^{\frac{\langle\chi,a_{\Gamma/\Gamma_v}\rangle}{\langle\chi,a_\Gamma\rangle}}=\lvert H\rvert^{\frac{\langle\chi,\chi_K\rangle}{\langle\chi,a_\Gamma\rangle}}.\]
If we denote by $\varphi$ a fixed absolutely irreducible character contained in $\chi$ and let $\{\varphi_1,\dots,\varphi_j\}$ be the set of all the distinct conjugates of $\varphi$, then
 \[\chi=d\sum_{i=1}^{j}\varphi_i.\]
 where $d$ is the Schur index. 
 So we have
 \[\langle\chi,\chi_K\rangle=d\sum_{i=1}^j\langle\varphi_i,\chi_K\rangle=dj\langle\varphi,\chi_K\rangle.\]
On the other hand, since the character $\varphi$ is absolutely irreducible,
 \[\langle\chi,a_\Gamma\rangle=d\sum_{i=1}^j\langle\varphi_i,a_\Gamma\rangle=dj\varphi(1)=djh\]
 where $h$ is the $h_i$ of Section~\ref{S:ssNot}, and one can check $h=\dim \varphi$.
We then know that
 \[\prod_{v|\infty}\lvert H^{\Gamma_v}\rvert=\lvert H\rvert^{\frac{\langle\chi,\chi_K\rangle}{\langle\chi,a_\Gamma\rangle}}=|H|^{\frac{1}{h}\langle\varphi,\chi_K\rangle}=\lvert H\rvert^u=\lvert H\rvert^{\uu}\]
completing the proof.
\end{proof}

\begin{remark}
Actually the statement of Theorem~\ref{thm:u in Cohen-Martinet} 
can be extended to some primes dividing $\lvert\Gamma\rvert$. Let $e$ be a central idempotent in $\Q[\Gamma]$ such that $e_1\cdot e=0$ and $S$ be a set of primes such that $e\in\Z_S[\Gamma]$ and $e\Z_S[\Gamma]$ is a maximal order in $e\Q[\Gamma]$ (i.e. $S$ only contains \emph{good primes} for $e$, see the definition in Section~\ref{S:NG}). If $H$ is a finite $e\Z_S[\Gamma]$-module, then
 \[|H|^{\uu}=\prod_{v|\infty}|H|^{\Gamma_v}.\]
The proof is  the same as above because Theorem 7.3 in \cite{CM90} still holds in this case.
\end{remark}

\section{Probabilities inversely proportional to automorphisms}\label{S:getaut}

Since the Cohen-Lenstra and Cohen-Martinet conjectures are rooted in the philosophy that objects appear inversely proportional as often as their number of automorphisms, it is natural to ask why there is a term $|G|^{\uu}$ in the conjectures at all.  One answer is that it was necessary to match computational evidence, and other heuristic explanations are given in \cite[Section 8]{CL84}.  In this section, we give another perspective, over the base field $\Q$, in which we see class groups as a part of a larger structure where $|G|^{\uu}|\Aut(G)|$ is the number of automorphisms of the larger structure.
 Bartel and Lenstra \cite{Bartel2018} have given a different perspective on interpreting these probabilities, over a general number field, as inversely proportional to the automorphisms of a larger object, in their case, the Arakelov class groups.  In contrast, our larger objects below are only  slightly larger than the class groups, and in particular, finite.

We choose an embedding $\bar{\Q}\sub\C$ so that $\Gal(\bar{\Q}/\Q)$ has a canonical decomposition group $\Gal(\C/\R)$ at $\infty$.
 We fix a map $s:\Gal(\C/\R)\to\Gamma$, let $K\sub \bar{\Q}$ be a Galois extension of $\Q$, let $\Gamma=\Gal(K/\Q)$, and let the decomposition group at $\infty$ given by $s$.
Let $K'$ be the maximal unramified abelian extension of $K$ in $\bar{\Q}$ of order prime to $|\Gamma|$.  
The structure we consider  is the finite group $G:=\Gal(K'/\Q)$ with given maps
$$
c: \Gal(\C/\R) \ra G \quad \quad \textrm{and} \quad \quad \pi: G \ra \Gal(K/\Q)=\Gamma,
$$
where $\pi$ is a surjection with abelian kernel.  Of course, $\ker(\pi)=\Cl^S_K$ (where $S$ is the set of primes not dividing $|\Gamma|$) is naturally a $\Gamma$-module, but  the data $(G,c,\pi)$ is a little more.  In fact, it is a \emph{class triple} as defined below.



\begin{definition} For a given map $s:\Gal(\C/\R)\to\Gamma$, we call $(G,c,\pi)$ a \emph{class triple} (for $s$) if $G$ is a finite group satisfying the following conditions:

\begin{enumerate}[i)]

\item  $\pi:G\to\Gamma$ is a surjective homomorphism such that $\ker\pi$ is an abelian group whose order is coprime to $|\Gamma|$;

\item $c:\Gal(\C/\R) \to G$ is a homomorphism such that $\pi\circ c=s$;

\item $\ker\pi^{\Gamma}=1$ (where $\Gamma$ acts by conjugation by preimages in $G$);

\item $\im c \cap \ker \pi =1$.

\end{enumerate}
Then for two class triples $(G_1,c_1,\pi_1)$ and $(G_2,c_2,\pi_2)$, a morphism $\tau$ is a group homomorphism $G_1\to G_2$ such that $\pi_1=\pi_2\circ\tau$ and that $\tau\circ c_1=c_2$.  
\end{definition}

\begin{theorem}
For a given map $s:\Gal(\C/\R)\to\Gamma$ and a class triple $(G,c,\pi)$, we have
$$
\lvert\Aut(G,c,\pi)\rvert=\lvert\ker\pi^{\im(s)}\rvert \lvert\Aut_\Gamma(\ker\pi)\rvert.
$$
Further, given a finite $\Gamma$-module $H$ of order relatively prime to $|\Gamma|$ with $H^{\Gamma}=1$, there is a unique 
isomorphism class of class triples for $s$ with $\ker \pi$ isomorphic to $H$ as a $\Gamma$-module.
\end{theorem}
\begin{proof}

Let $A$ be the group of  automorphisms of $(G,c,\pi)$, and since each such automorphism preserves $\ker \pi$ (set-wise) and respects  $\pi$, 
we have a homomorphism
$$
A \ra \Aut_\Gamma(\ker\pi).
$$
By the Schur-Zassenhaus theorem, we can write $G=\ker \pi \rtimes \Gamma$ (non-canonically), and so in this notation an element $\tau\in A$ is determined by where it sends $\ker \pi$ and $\Gamma$.  Further, since $\pi=\pi \circ \tau$, it follows that $\tau$ sends $\Gamma$ to another splitting of $G\ra\Gamma$.  By Schur-Zassenhaus all the splittings are conjugate by elements of $\ker \pi$.

So we have a map from $\ker \pi$ to the splittings $\Gamma \ra G$.  We claim this give $|\ker \pi|$ distinct splittings.
We have
$$
(n,1)(1,\gamma)(n,1)^{-1}=(n(n^{-1})^{\gamma^{-1}},\gamma).
$$
Suppose that for all $\gamma\in \Gamma$, we have
$$
n_1(n_1^{-1})^{\gamma^{-1}}
=n_2(n_2^{-1})^{\gamma^{-1}}, 
$$
i.e., $n_2^{-1}n_1=(n_2^{-1}n_1)^{\gamma^{-1}}$.
By the definition of class triple, this implies $n_1=n_2$ and so we have $|\ker \pi|$ splittings.  

Any element $\Aut_\Gamma(\ker\pi)$ and any splitting $\Gamma \ra H$ combine to give an automorphism of $(G,\pi)$ by the definition of semi-direct product.
We next determine which of these automorphisms preserves $c$. 
Let $K\sub G$ be $K:=\pi^{-1}(\im \pi \circ c)$.  So we have
$$
1\ra \ker \pi \ra K \ra \im \pi \circ c \ra 1.
$$
Since $\im c \cap \ker \pi =1$, one splitting of the above is $\im \pi \circ c\ra \im c.$
Another splitting is $\im \pi \circ c\ra 1 \times \im \pi \circ c \sub \ker \pi \rtimes \Gamma$ according to our chosen splitting above.
By Schur-Zassenhaus, these two splittings are conjugate by an element $(n,1)$ for some $n\in \ker\pi$.

So let $I=\im \pi \circ c$.  Then the elements of $\im c$ are $(n,1)(1,\gamma)(n^{-1},1)=
(n(n^{-1})^{\gamma^{-1}},\gamma)$ for $\gamma\in I$.  
These elements are fixed by the element of $\Aut(G,\pi)$ that comes from $\psi\in \Aut_\Gamma(\ker\pi)$ and conjugation of $\Gamma$ by $(m,1)$ if and only if for all $\gamma\in I$, 
$$
(m,1)(\psi(n(n^{-1})^{\gamma^{-1}}),\gamma)(m^{-1},1)=(n(n^{-1})^{\gamma^{-1}},\gamma)
$$
i.e. 
$$
n^{-1}m\psi(n)=(n^{-1}m\psi(n))^{\gamma^{-1}}
$$
i.e.
$n^{-1}m\psi(n)$ is fixed by $I$, i.e $m\in n^{-1}(\ker\pi)^{I} \psi(n)$.
Thus we conclude that exactly $|\Aut_\Gamma(\ker\pi)||(\ker\pi)^{I}|$ elements of $\Aut(H,\pi)$ preserve $c$, which proves the first statement of the theorem.

For the second statement of the theorem, by Schur-Zassenhaus, any class triple giving $H$ has $G\isom H\rtimes \Gamma$.  Choosing $c$ to be $s$ composed with the trivial splitting $\Gamma \ra H\rtimes \Gamma$ gives at least one class triple giving $H$.  As we saw above, any other choice of $c$ differs by
conjugation by an element of $H$, i.e. differs by an automorphism of $H\rtimes \Gamma$ fixing the map to $\Gamma$.
\end{proof}

\begin{corollary}
Let $K\sub \bar{\Q}$ be a Galois extension of $\Q$
 with Galois group $\Gamma$ and decomposition group  $\Gamma_\infty$ at $\infty$ and map $s:\Gal(\C/\R)\to \Gamma_\infty \sub \Gamma $. 
Let $G:=\Gal(K'/\Q)$ with given maps
$$
c: \Gal(\C/\R) \ra G \quad \quad \textrm{and} \quad \quad \pi: G \ra \Gal(K/\Q)=\Gamma,
$$ Let $S$ be the set of primes not dividing $|\Gamma|$.
Then 
$$
|\Aut(G,c,\pi)|=|(\Cl_K^{S})^{\Gamma_\infty}||\Aut_\Gamma(\Cl_K^{S})|.
$$ 
\end{corollary}

So, combining with Theorem~\ref{thm:u in Cohen-Martinet}, we see that the probabilities in the Cohen-Lenstra and Cohen-Martinet conjectures are inversely proportional to the number of automorphisms of the class triples associated to the fields (which are determined up to isomorphism by their class groups and decomposition groups but have a different number of automorphisms from their class groups). 

\further{
\begin{example}
As a concrete special case, consider quadratic extensions and how often the class group is $\Z/3\Z$.  We see that $S_3$ has 6 auts when there is no order $2$ element fixed, but 2 auts with an order $2$ element fixed.  Comparing this to trivial class group, $C_2$ has 1 aut with or without an order $2$ element fixed, so we see the extra factor of $3$ here in the totally real case, where there is no order $2$ element to fix.  
\end{example}}

\section{Moments of the Cohen-Lenstra-Martinet Random Groups}\label{S:moments}

In this section, we will find the moments of the Cohen-Lenstra-Martinet random $\Gamma$-modules, and moreover show that their distributions are determined by their moments.

\subsection{Moments for Galois Extensions}
We keep the notation from Section~\ref{S:ssNot}. However, in this section, we will take the set $S$ of prime to be not necessarily finite.  We will also define a slightly more general notion of random modules.

\begin{definition}(Random $\fO$-modules) Let $A$ be any finite dimensional semisimple $\Q$-algebra with $m$ simple factors. Let $S$ be a set of prime numbers, $\fO$ be a $\Z_S$-maximal order of $A$, and $\uu\in\Q^m$ be a fixed $m$-tuple. If either $S$ contains finitely many primes or $u_i>0$ for all $i=1,\ldots,m$, then we define $X=X(A,\uu,\fO)$ to be a random finite $\fO$-module such that for all finite $\fO$-module $G_1$ and $G_2$, we have
 \[\frac{\mathbb{P}(X\cong G_1)}{\mathbb{P}(X\cong G_2)}=\frac{|G_2|^{\uu}|\Aut_\fO(G_2)|}{|G_1|^{\uu}|\Aut_\fO(G_1)|}.\]
\end{definition}

When $S$ does not contain any primes dividing $|\Gamma|$, then $\Z_S[\Gamma]$ is a maximal order in $\Q[\Gamma]$ (and so $(1-e_1)\Z_S[\Gamma]$ is a maximal order in $(1-e_1)\Q[\Gamma]$), and our previous definition of $X$ is a special case of the above.
As in Remark~\ref{remark:class groups have trivial invariant part},  $X$ is well-defined.

Now given $H$ a finite $\fO$-module, consider the function $|\operatorname{Sur}_\fO(G,H)|$ counting the number of surjective $\fO$-morphisms from $G$ to $H$. Then we have the following formula to compute the moments of $X$.

\begin{theorem}\label{thm: formula for moments} Given a finite $\fO$-module $H$, we have
 \[\mathbb{E}\left(\lvert\Sur_\fO(X,H)\rvert\right)=\frac{1}{|H|^{\uu}}.\]
\end{theorem}

\begin{proof} 
In this proof a summation over $G/\sim$ always means the sum is over all isomorphism classes of finite $\fO$-modules, with $G$ a representative from each class.
For finite $\fO$-modules $G,H$, we have
 \begin{equation*}\lvert\operatorname{Sur}_\fO(G,H)\rvert=\#\{G'\subset G|\,G/G'\cong H\}\cdot\lvert\Aut_\fO(H)\rvert.\end{equation*}
 where $G'\sub G$ denotes $G'$ a sub-$\fO$-module of $G$.
For $G_1$ and $G_2$ finite $\fO$-modules, \cite[Proposition 3.3]{CM90} gives

\begin{equation*}\sum_{G/\sim}\lvert\Aut_\fO(G)\rvert^{-1}\#\{H\subseteq G:H\cong G_1\text{ and }G/H\cong G_2\}=\lvert\Aut_\fO(G_1)\rvert^{-1}\lvert\Aut_\fO(G_2)\rvert^{-1}.\end{equation*}
Let 
\begin{equation}\label{E:defZ}
Z(\uu)=\sum_{G/\sim}  \frac{1}{|G|^{\uu}\lvert\Aut_\fO(G)\rvert}.
\end{equation}
 Then we deduce that
\begin{equation*}\begin{aligned}
&\mathbb{E}\big(\lvert\Sur_\fO(X,H)\rvert\big)=\sum_{G/\sim}\mathbb{P}(X\cong G)\lvert\Sur_\fO(G,H)\rvert\\
&=\sum_{G/\sim}\frac{1}{\lvert G\rvert^{\uu} \lvert\Aut_\fO(G)\rvert Z(\uu)}\lvert\Aut_\fO(H)\rvert\sum_{G_1/\sim}\#\{G'\subseteq G|\,G'\cong G_1,\,G/G'\cong H\}\\
&=\frac{\lvert\Aut_\fO(H)\rvert}{Z(\uu)}\sum_{G_1/\sim}\frac{1}{\lvert G_1\rvert^{\uu} \lvert H\rvert^{\uu}  } \sum_{G/\sim} \frac{1}{\lvert\Aut_\fO(G)\rvert}  \#\{G'\subseteq G|\,G'\cong G_1,\,G/G'\cong H\}\\
&=\lvert\Aut_\fO(H)\rvert\sum_{G_1/\sim}\frac{1}{\lvert\Aut_\fO(G_1)\rvert\cdot\lvert G_1\rvert^{\uu}Z(\uu)}\frac{1}{\lvert\Aut_\fO(H)\rvert\cdot\lvert H\rvert^{\uu}}\\
&=\frac{1}{\lvert H\rvert^{\uu}}\sum_{G_1/\sim}\mathbb{P}(X\cong G_1)=\frac{1}{\lvert H\rvert^{\uu}}.
\end{aligned}\end{equation*}
\end{proof}

When applying the results to class groups, it is always the case that we only consider the $e$-component of $\Q[\Gamma]$ where $e$ is some central idempotent. Suppose that $e$ is some central idempotent in $A=\Q[\Gamma]$, then $eA\subseteq\Q[\Gamma]$ is also a semisimple $\Q$-algebra and $e\fO$ is a maximal order in $eA$.  We could build a random module directly from $e\fO$, or we could multiply our original random module by $e$.  The following shows these two constructions are the same.

\begin{lemma}\label{lemma:5.3} Let $e=e_2+\cdots+e_k$ be some central idempotent of $A$, and let $X_1=X(A,\uu=(u_1,\dots,u_m),\fO)$ and $X_2=X(eA,\uv=(v_2,\dots,v_k),e\fO)$ be the random modules defined in Section~\ref{S:Notations} such that $u_i=v_i$ for all $i=2,\dots,k$. Then $eX_1$ and $X_2$ have the same probability distribution, i.e., for all finite $e\fO$-modules $G$, we have
 \[\mathbb{P}(eX_1\cong G)=\mathbb{P}(X_2\cong G).\]
\end{lemma}

\begin{proof}
Let $\mathcal{S}$ be the set of isomorphism classes of finite $(1-e)\fO$-modules. For all finite $e\fO$-modules $G_1,G_2$, we have
 \[\begin{aligned}
 \frac{\mathbb{P}(eX_1\cong G_1)}{\mathbb{P}(eX_1\cong G_2)}&=\frac{\sum_{H\in\mathcal{S}}\mathbb{P}(X_1\cong G_1\oplus H)}{\sum_{H\in\mathcal{S}}\mathbb{P}(X_1\cong G_2\oplus H)}
 \end{aligned}\] 
Since all the terms defining the probabilities factor over $G_i$ and $H$, we conclude the lemma.
\end{proof}

Therefore Theorem~\ref{thm: formula for moments} can be applied to $eX$ directly.
\begin{corollary}\label{cor:Galoismoments} Let $e\in\fO$ be any central idempotent. Given a finite $\fO$-module $H$, we have
\begin{equation*}
\mathbb{E}\left(|\Sur_\fO(eX,H)|\right)=\left\{\begin{aligned}
&\frac{1}{|H|^{\uu}}\ &\text{if }eH=H,\\
&0\ &\text{otherwise.}
\end{aligned}\right.
\end{equation*}
\end{corollary}

\begin{proof}  If $eH\neq H$, then there is no surjective homomorphism from any $e\fO$-module to $H$. If $eH=H$, then $\fO$-morphisms from $eG$ to $H$ are the same as $e\fO$-module homomorphisms from $eG$ to $H$.  So the corollary follows from Lemma \ref{lemma:5.3} and Theorem \ref{thm: formula for moments}. 
\cut{
Let $e=e_2+\cdots+e_k$ be the decomposition of $e$.
So we assume that $eH=H$. Then $G=eG\oplus(1-e)G$, and any $\fO$-morphism from $(1-e)G$ to $H$ is just $0$ by our assumption. So the image of any morphism $\varphi:G\to H$ is the same as $\varphi\big|_{eG}:eG\to H$, and the restriction induces an isomorphism of abelian groups of homomorphisms $\Hom_\fO(G,H)\to\Hom_{e\fO}(eG,H)$. In particular,
 \[\lvert\Sur_\fO(G,H)\rvert=\lvert\Sur_{e\fO}(eG,H)\rvert.\]
Then by Lemma \ref{lemma:5.3}, we can apply Theorem \ref{thm: formula for moments} to $(eA,(u_2,\dots,u_k),e\fO)$ and obtain
 \[\mathbb{E}\left(\lvert\Sur_\fO(eX,H)\rvert\right)=\sum_{G/\sim}\lvert\Sur_{e\fO}(eG,H)\rvert\mathbb{P}(eX\cong eG)=\frac{1}{\lvert H\rvert^{(u_2,\dots,u_k)}}=\frac{1}{\lvert H\rvert^{\uu}}.\]
hence the result.}
\end{proof}

Now we will show that the expected values of functions of $X$ agree with the averages
that appear in the conjectures of \cite{CM90}.

\begin{remark}
The original definition of $M_{\uu}^S(f)$, the average appearing the the conjectures in \cite{CM90}, is given by their Definition 5.1 and Conjecture 6.6. However, note that in the original paper,  the definition of $M^S_{\uu}(f)$ must be corrected to involve $e$, e.g. $M^S_{\uu}(f)$ should be defined with the implicit algebra $e\Q[\Gamma]$ instead of $\Q[\Gamma]$.
\end{remark}

\begin{proposition}\label{Prop: M(f)=E(f(X))} Let $|S|<\infty$, and let $f$ be a  non-negative function defined on the isomorphism classes of finite $\fO$-modules. For $X=X(A,\uu,\fO)$, we have
$$\mathbb{E}\big(f(X)\big)=\lim_{\underline{x}\to\ui}\frac{\sum_{|G|\leq\underline{x}}|G|^{-\uu}\sum_{\varphi\in\Hom(P,G)}\lvert\Aut_\fO(G)\rvert^{-1}f(G/\operatorname{Im}\varphi)}{\sum_{|G|\leq\underline{x}}|G|^{-\uu}\sum_{\varphi\in\Hom(P,G)}\lvert\Aut_\fO(G)\rvert^{-1}}$$
where
the sum is over finite $\fO$-modules $G$ and
 $P$ is a projective $\fO$-module of rank $\uu$ (as defined in \cite[Definition 3.1]{CM90}).
 Here $\underline{x}\in\Z^m$, and $|G|\leq  \underline{x}$ means that for every $i$, we have $|e_iG|\leq x_i$, and the limit means all $x_i\ra\infty$.
\end{proposition}

\begin{proof} 

In this proof a summation over $G/\sim$ always means the sum is over all isomorphism classes of finite $\fO$-modules, with $G$ a representative from each class.
By \cite[Theorem 4.6 (ii)]{CM90} with $\psi(G)=|\Aut_{\fO}(G)|^{-1}$ and $\us=\uu$, if $g_{G_1}(G)=\#\{\varphi\in\Hom_\fO(P,G):G/\im\varphi\cong G_1\}$ and $P$ is projective of rank $\uu$, then 
 \[
 \sum_{G/\sim}\frac{g_{G_1}(G)}{\lvert\Aut_\fO(G)\rvert\lvert G\rvert^{\uu}}=\frac{Z(\underline{0})}{\lvert\Aut_\fO(G_1)\rvert\lvert G_1\rvert^{\uu}Z(\uu)},\]
 where $Z$ is defined in \eqref{E:defZ} (and see Remark~\ref{remark:class groups have trivial invariant part} for the convergence). Then we have
 \[\begin{aligned}
 &\sum_{G/\sim}|G|^{-\uu}\sum_{\varphi\in\Hom_\fO(P,G)}\lvert\Aut_\fO(G)\rvert^{-1}f(G/\im\varphi)\\
 &=\sum_{G_1/\sim}f(G_1)\sum_{G/\sim}\frac{g_{G_1}(G)}{\lvert\Aut_\fO(G)\rvert\lvert G\rvert^{\uu}}\\
 &=\sum_{G_1/\sim}f(G_1)\frac{Z(\underline{0})}{\lvert\Aut_\fO(G_1)\rvert\lvert G_1\rvert^{\uu}Z(\uu)}=Z(\underline{0})\mathbb{E}\big(f(X)\big)
 \end{aligned}\]
We can also apply this to the constant function $f(G)=1$, and deduce the proposition.
\end{proof}

\subsection{Moments Determine the Distribution}
So the random $\fO$-module $X$ has $H$-moment $|H|^{-\uu}$ for every finite $\fO$-module $H$. Now we ask:  given a random finite $\fO$-module $Y$ with $H$-moment $|H|^{-\uu}$ for all $H$, does $Y$ have the same probability distribution as $X$?  In this section, we will see the answer is yes.

Recall the notations from Section~\ref{S:ssNot}: 
$A=\prod_{i=1}^m A_i$ and $K_i$ is the center of $A_i$. Now for each pair $(i,\fp)$ where $i=1,\dots,m$ and $\fp$ is a prime of $K_i$, we can consider the completion $A_{i,\fp}\cong M_{l_{i,\fp}}(D_{i,\fp})$ of $A_i$ at $\fp$
(where $D_{i,\fp}$ is the completion of $D_i$ at $\fp$ and $l_{i,\fp}$ is some positive integer). If $\fO$ is a maximal $\Z_S$-order in $A$, then $e_i\fO$ also admits a completion $\fO_{i,\fp}= e_i\fO\otimes_{\Z_{K_i}}\Z_{K_{i,\fp}}$
(where $\Z_{K_i}$ is the ring of integers of $K_i$ and $\Z_{K_{i,\fp}}$ is the valuation ring of
$K_{i,\fp}$). In particular, $\fO_{i,\fp}$ is a maximal order in $A_{i,\fp}$. Then in this case (unlike in the global case), there always exists an isomorphism
 \[\fO_{i,\fp}\cong M_{l_{i,\fp}}(\so_{i,\fp}),\]
where $\so_{i,\fp}$ is \emph{the} maximal order in $D_{i,\fp}$, which is given by a valuation.

If $G$ is a finite $\fO$-module, and $(i,\fp)$ some prime ideal of $\fO$ (i.e. $\fp$ is a prime ideal of $K_i$), then let $G_\fp$ denote the part of $G$ annihilated by a power of $\fp$ and we know that $G_\fp$ is naturally a finite $\fO_{i,\fp}$-module. For any two finite $\fO$-modules $G_1$ and $G_2$, we have
 \[|\Aut_\fO(G_1)|=\prod_{(i,\fp)}|\Aut_{\fO_{i,\fp}}(G_{1,\fp})|\quad \text{and}\quad |\Sur_\fO(G_1,G_2)|=\prod_{(i,\fp)}|\Sur_{\fO_{i,\fp}}(G_{1,\fp},G_{2,\fp})|.\]
Moreover, the category of $\fO_{i,\fp}$-modules is equivalent to the category of $\so_{i,\fp}$-modules, because they are both matrix algebras over $\so_{i,\fp}$. So the question of counting surjective morphisms is then reduced to the following case: let $D$ be a division algebra over $\Q_p$ with the maximal $\Z_p$-order $\so$ and we consider the category of finite $\so$-modules. Given any (finite) partition $\lambda:\lambda_1\geq\lambda_2\geq\dots$, there exists a unique (up to isomorphism) finite $\so$-module $G$ such that
 \[G\cong\bigoplus_i\so/\fp^{\lambda_i},\]
where $\fp$ is the unique maximal ideal of $\so$, see, e.g. \cite[Lemma 2.7]{CM90}. Then we write $G=G_\lambda$ and call it the $\so$-module of type $\lambda$. Also let $q=|\so/\fp|$ be the cardinality of the simple $\so$-module.

\begin{definition}
Given a partition $\lambda:\lambda_1\geq\lambda_2\geq\cdots\geq\lambda_n$, it can be represented by a Young diagram whose number of boxes in the $i$th row represents the number $\lambda_i$. Then the transpose $\lambda'$ of $\lambda$ is the partition such that $\lambda_j'$ equals to the number of boxes in the $j$th column in the diagram of $\lambda$.
\end{definition}

\begin{lemma}\label{lemma: number of hom}
Let $D$ be a division algebra over $\Q_p$ with maximal $\Z_p$-order $\so$. Given two $\so$-modules $G_\lambda,G_\mu$ of type $\lambda$ and $\mu$. Then 
 \[|\Hom_{\scriptscriptstyle\mathcal{O}}(G_\lambda,G_\mu)|=q^{\sum_{i=1}^\infty\lambda_i'\mu_i'}.\]
\end{lemma}

\begin{proof}
By Lemma 2.7 (and more generally \S2) in [CM90], we only need to check the formula for the case when $G_\lambda,G_\mu$ are both cyclic, which is clear, i.e.,
 \[|\Hom_{\scriptscriptstyle\mathcal{O}}(\so/\fp^m,\so/\fp^n)|=q^{\min(m,n)}=q^{\lambda'_1\mu'_1}.\]
\end{proof}

\begin{lemma}\label{L:oftype} Let $G=G_\lambda$ be a $\so$-module of type $\lambda$. If $\mu\leq\lambda$, then the number of submodules of type $\mu$, denoted by $\alpha_\lambda(\mu;q)$, satisfies
 \[\alpha_\lambda(\mu;q)\leq\prod_{j\geq1}\frac{1}{(1-2^{-j})^{\lambda_1}}\cdot q^{\sum_{i=1}^{\lambda_1}\mu_i'\lambda_i'-(\mu_i')^2}.\]
\end{lemma}

\begin{proof}
First we claim
 \[\alpha_\lambda(\mu;q)\leq\frac{|\Hom_{\scriptscriptstyle\mathcal{O}}(G_\mu,G_\lambda)|}{|\Aut_{\scriptscriptstyle\mathcal{O}}(G_\mu)|},\]
i.e., if $f:G_\mu\to G_\lambda$ happens to be an injective map, then $f\circ g$ where $g\in\Aut_{\scriptscriptstyle\mathcal{O}}(G_\mu)$ clearly gives us the same subgroup in $G_\lambda$. Then by Theorem 2.11 in [CM90], if $\pi_1,\dots,\pi_t$ are the distinct (nonzero) values of $\{\mu_i\}$ with multiplicities $k_1,\dots,k_t$, then
 \[\lvert\Aut_{\scriptscriptstyle\mathcal{O}}(G_\mu)\rvert=q^{\sum_i(\mu'_i)^2}\prod_{i=1}^t(k_i)_q\geq q^{\sum_i(\mu'_i)^2}\prod_{i=1}^t(\infty)_q\geq q^{\sum(\mu_i')^2}\prod_{j=1}^\infty(1-q^{-j})^{\mu_1},\]
where the notion $(k)_q$ means $\prod_{i=1}^k(1-q^{-i})$ if $k>0$. Since $\mu_1\leq\lambda_1$, we have
 \[\alpha_\lambda(\mu;q)\leq\frac{\lvert\Hom_{\scriptscriptstyle\mathcal{O}}(G_\mu,G_\lambda)\rvert}{\lvert\Aut_{\scriptscriptstyle\mathcal{O}}(G_\mu)\rvert}\leq\prod_{j=1}^\infty\frac{1}{(1-q^{-j})^{\mu_1}}q^{\sum\mu_i'\lambda_i'-(\mu_i')^2}\leq\prod_j\frac{1}{(1-2^{-j})^{\lambda_1}}\cdot q^{\sum\mu_i'\lambda_i'-(\mu_i')^2}.\]
\end{proof}

\begin{lemma}\label{lemma: estimation of Hom-moments}
For any given $\so$-module $G$ of type $\lambda$, there exists a constant $C$ such that 
 \[\#\{H\subseteq G\}\leq C^{\lambda_1}q^{\frac{1}{4}\sum(\lambda_i')^2}.\]
\end{lemma}

\begin{proof}
To prove this lemma, we sum the result in Lemma~\ref{L:oftype} over all $\mu$, and a bound for this sum is given in \cite[Lemma 7.5]{MW17}.
\end{proof}

Now using the lemmas above and results from \cite{MW17}, we can prove that the Cohen-Lenstra-Martinet distributions are determined by their moments, and in fact even a sequence of random variables with moments converging to their moments must converge to the Cohen-Lenstra-Martinet distribution.

\begin{theorem}\label{thm: limit moments}
Take $A,\fO,m$ as in Section~\ref{S:ssNot} and let $\uu\in\Q^m$ be an $m$-tuple. 
Assume that either that $|S|<\infty$ and $\uu\geq\underline{0}$, or, that $|S|=\infty$ and $u_i>0$ for all $i$.
Let \(K_i\) be the center of each component \(A_i\) and \(R_i\) the integral closure of \(\Z_S\) in \(K_i\).
Then \(R:=\bigoplus R_i\) is the center of \(\fO\) and each \(\fO_i\) is a maximal \(R_i\)-order in \(A_i\) (see \cite[Theorem 10.5]{R03maximal}).

Let $\{X_n\}$ be a sequence of random variables taking values in finite $\fO$-modules. 
For each prime $\fp$ of $\fO$, let $n_\fp\geq 0$ such that $n_\fp=0$ for almost all $\fp$.
Let $\mathcal{S}$ be the set of all finite $\fO$-modules $H$ such that the annihilator of $H_{\fp}$ divides $\fp^{n_\fp}$.  Moreover let $N$ be the $\fO$-module such that $N_{\fp}$ is of type $(n_\fp,0,0,\dots)$.

Suppose that for every $G\in\mathcal{S}$, we have
 \[\lim_{n\to\infty}\mathbb{E}\big(\lvert\Sur_\fO(X_n,G)\rvert\big)=\frac{1}{|G|^{\uu}}.\]
Then for every $H\in\mathcal{S}$, the limit 
 \[\lim_{n\to\infty}\mathbb{P}(X\otimes_R N\cong H)\]
exists and for all $G\in\mathcal{S}$ we have
 \[\sum_{H\in\mathcal{S}}\lim_{n\to\infty}\mathbb{P}(X\otimes_R N\cong H)\lvert\Sur_\fO(H,G)\rvert=\frac{1}{|G|^{\uu}}.\]

Suppose $\{Y_n\}$ is another sequence of random variables taking values in finite $\fO$-modules such that for every $G\in\mathcal{S}$, we have
 \[\lim_{n\to\infty}\mathbb{E}\big(\lvert\Sur(Y_n,G)\rvert\big)=\frac{1}{|G|^{\uu}}.\]
Then for every $H\in A$, we have
 \[\lim_{n\to\infty}\mathbb{P}(X\otimes_R N\cong H)=\lim_{n\to\infty}\mathbb{P}(Y\otimes_R N\cong H).\]
\end{theorem}

\begin{proof} The proof is very similar to \cite[Theorem 8.3]{MW17}, so we only present a sketch and highlight the differences.
First we suppose that the limit
 \[\lim_{n\to\infty}\mathbb{P}(X_n\otimes_R N\cong H)\]
exists for all $H\in\mathcal{S}$ and we are going to show that for all $G\in\mathcal{S}$ we have
 \[\sum_{H\in\mathcal{S}}\lim_{n\to\infty}\mathbb{P}(X_n\otimes_R N\cong H)|\Sur_\fO(H,G)|=\frac{1}{|G|^{\uu}}.\]
By Lemma \ref{lemma: number of hom} and  the same argument as in \cite[Theorem 8.3]{MW17}, for each $G\in\mathcal{S}$, there exists $G'\in\mathcal{S}$ such that 
 \[\sum_{H\in\mathcal{S}}\frac{\lvert\Hom_\fO(H,G)\rvert}{\lvert\Hom_\fO(H,G')\rvert}<\infty.\]

Then the same argument as in in \cite[Theorem 8.3]{MW17} using the Lebesgue Dominated Convergence Theorem  concludes that
\[\begin{aligned}
 &\sum_{H\in\mathcal{S}}\lim_{n\to\infty}\mathbb{P}(X_n\otimes_R N\cong H)\lvert\Sur(H,G)\rvert\\
 &=\lim_{n\to\infty}\sum_{H\in\mathcal{S}}\mathbb{P}(X_n\otimes_R N\cong H)\lvert\Sur(H,G)\rvert=\frac{1}{|G|^{\uu}}
 \end{aligned}\]
i.e., if for all $H\in\mathcal{S}$ the limit $\lim_{n\to\infty}\mathbb{P}(X_n\otimes_R N\cong H)$ exists, then the moments agree with $\mathbb{E}\big(\lvert\Sur(X,G)\rvert\big)$ for all $G\in\mathcal{S}$.

Next we show that if the limits $\lim_{n\to\infty}\mathbb{P}(X_n\otimes_R N\cong H)$ and $\lim_{n\to\infty}\mathbb{P}(Y_n\otimes_R N\cong H)$ exist for all $H$ then
 \[\begin{aligned}
 &\sum_{H\in\mathcal{S}}\lim_{n\to\infty}\mathbb{P}(Y_n\otimes_R N\cong H)\lvert\Sur(H,G)\rvert
 &=\sum_{H\in\mathcal{S}}\lim_{n\to\infty}\mathbb{P}(X_n\otimes_R N\cong H)\lvert\Sur(H,G)\rvert=\frac{1}{|G|^{\uu}}
 \end{aligned}\]
implies 
 \[\lim_{n\to\infty}\mathbb{P}(Y_n\otimes_R N\cong H)=\lim_{n\to\infty}\mathbb{P}(X_n\otimes_R N\cong H).\]
Note that the averages $|\Hom_\fO(X,H)|$ and $|\Sur_\fO(X,H)|$ over all $H$,  are determined from one another by finitely many steps of addition and subtraction. We'll apply \cite[Theorem 8.2]{MW17} with distinct primes $p_i$'s in the assumption replaced by not necessarily distinct real numbers $q_i$'s. The proof of the theorem actually proves the statement in this generality.

Now let $M$ be the set defined in \cite[Theorem 8.2]{MW17} where the choice of $q_{i}$ comes from the following: there are only finitely many primes $\fp_{ij}\subseteq\Z_{K_i}$ such that $n_{\fp_{ij}}>0$ for all $i=1,\dots,m$, so we can  let $q_k=|\so_k/\fp_k'|$ where $\so_k\subseteq D_{i,\fp_k}$ is the maximal order in $D_{i,\fp_k}$ and $\fp_k'$ is the unique maximal ideal. We say that an $\fO$-module $G\in\mathcal{S}$ corresponds to $\mu\in M$ if the type of $G$ is exactly $\mu'$ where $\mu'$ is obtained by $(\mu')_k=(\mu_k)'$. We then define
 \[x_\mu=\lim_{n\to\infty}\mathbb{P}(X_n\otimes_R N\cong G_{\mu'})\]
for all $\mu\in M$. And similarly for $y_\mu$. If we let $C_\lambda$ denote the expected value of the number of homomorphisms into $G_{\lambda'}$, then by Lemma \ref{lemma: estimation of Hom-moments}, we know that $C_\lambda$ satisfies the condition in \cite[Theorem 8.2]{MW17}. Then \cite[Theorem 8.2]{MW17} tells us that $x_\mu$ and $y_\mu$ are determined by $C_\lambda$.

Finally, the same diagonal argument at the end of the proof \cite[Theorem 8.3]{MW17} shows that when the limit moments are $|G|^{-\uu}$, the limit $\lim_{n\to\infty}\mathbb{P}(X_n\otimes_R N\cong H)$ exists for all $H\in\mathcal{S}$.
\end{proof}

The above theorem is the most flexible for applications, but we will state now simpler versions to emphasize the main point.

\begin{theorem}\label{thm: moments determine distribution finite case}
Keep the notations in Theorem~\ref{thm: limit moments}.
Assume that $|S|<\infty$. If $\{X_n\}$ is a sequence of random variables taking values in finite $\fO$-modules such that
 \[\lim_{n\to\infty}\mathbb{E}\big(|\Sur_\fO(X_n,G)|\big)=\frac{1}{|G|^{\uu}}\]
for all finite $\fO$-module $G$, then 
 \[\lim_{n\to\infty}\mathbb{P}(X_n\cong G)=\frac{1}{|\Aut_\fO(G)||G|^{\uu}Z(\uu)},\]
i.e., the limit of the random variables exists and has the same probability distribution as the random variable $X=X(A,\uu,\fO)$.
\end{theorem}

\begin{proof}
If $|S|<\infty$, we can take into account all the prime ideals of $\fO$ at one time. Provided that $G$ is a finite module such that $G_{i,\fp}$ is of type $\lambda^{i,\fp}$ where $\lambda^{i,\fp}$ is a partition, then in Theorem \ref{thm: limit moments} we take $n_{i,\fp}=(\lambda^{i,\fp})'_1+1$. If $H$ is any $\fO$-module such that
 \[H\otimes_R N\cong G,\]
then $H$ has to be isomorphic to $G$, i.e., $\mathbb{P}(X_n\cong G)=\mathbb{P}(X_n\otimes N\cong G)$, and it is determined by the limit moments.
\end{proof}

\begin{theorem} \label{T:infmom}
Assume that $|S|=\infty$ and $u_i>0$ for all $i=1,\dots,m$, and $X=X(A,\uu,\fO)$ is the random variable we've defined.
If $Y$ is a random variable taking values in finite $\fO$-modules such that 
 \[\mathbb{E}\big(|\Sur_\fO(Y,G)|\big)=\frac{1}{|G|^{\uu}}=\mathbb{E}\big(|\Sur_\fO(X,G)|\big).\]
Then
 \[\mathbb{P}(Y\cong G)=\mathbb{P}(X\cong G),\]
for all finite $\fO$-modules $G$.
\end{theorem}

\begin{proof}
We let $\fp_i$ be the primes of $\fO.$
By Theorem \ref{thm: moments determine distribution finite case}, for every $n$ we have
 \[\mathbb{P}(Y_{\fp_i}\cong G_{\fp_i}|\,i=0,1,\dots,n)=\mathbb{P}(X_{\fp_i}\cong G_{\fp_i}|\,i=0,1,\dots,n).\] Then by basic properties of measures, we have
\begin{equation*}\begin{aligned}
\mathbb{P}(Y\cong G)&=\mathbb{P}(Y_{\fp_i}\cong G_{\fp_i}|\,i=0,1,2,\dots)\\
&=\lim_{n\to\infty}\mathbb{P}(Y_{\fp_i}\cong G_{\fp_i}|\,i=0,1,\dots,n)\\
&=\lim_{n\to\infty}\mathbb{P}(X_{\fp_i}\cong G_{\fp_i}|\,i=0,1,\dots,n)\\
&=\mathbb{P}(X_{\fp_i}\cong G_{\fp_i}|\,i=0,1,2,\dots)=\mathbb{P}(X\cong G).
\end{aligned}\end{equation*}
\end{proof}

However the statement on limit moments determining the limit distributions does not hold if $S$ contains infinitely many primes.
\begin{example} Let $S$ contain infinitely many prime numbers which are relatively prime to $|\Gamma|$ (so that $\fO=\Z_S[\Gamma]$) and  $u_i>0$ for all $i$. Let $H$ be any finite $\fO$-module. Then   $\mathbb{P}(X\cong H)>0$. 

For every rational prime $p$, there is a $\fO$-module $G_p$ whose underlying abelian group is a $p$-group, say $(\Z_S/p\Z_S)^n\cong(\Z/p\Z)^n$ which is a representation of $\Gamma$ over the finite field $\mathbb{F}_p$. 
Let $Y_p$ be a random $\fO$-module such that
\begin{equation*}\mathbb{P}(Y_p\cong G)=\left\{\begin{aligned}
&\mathbb{P}(X\cong G)\ &\forall\,G\neq H\text{ or }H\times G_p;\\
&0\ &\text{if }G=H;\\
&\mathbb{P}(X\cong H)+\mathbb{P}(X\cong H\times G_p)\ &\text{if }G=H\times G_p.
\end{aligned}\right.\end{equation*}
Since $|\Sur_\fO(H,G)|=|\Sur_\fO(H\times G_p,G)|$ whenever $p>|G|$, for every $\fO$-module $G$, we have
$$
\lim_{p\ra\infty} \E(|\Sur_\fO(Y_p,G)| ) =\E(|\Sur_\fO(X,G)| ).
$$
However  $\lim_{p\ra\infty}\mathbb{P}(Y_p\cong H)=0$. This shows there is no analog of Theorem~\ref{thm: moments determine distribution finite case} for infinite $S$.
\end{example}

\section{Explanation of the Cohen-Martinet Heuristics in the non-Galois case}\label{S:NG}

Cohen and Martinet \cite{CM90} do not specifically make a conjecture about the distribution of class groups of non-Galois fields.  However, they do show that by expressing class groups of non-Galois fields in terms of Galois fields, such conjectures can be obtained as consequences of their conjectures in some cases.  The goal of this section is to deduce the entire consequence of the Cohen-Martinet conjectures for class groups of non-Galois fields.  Interestingly, in the non-Galois case, one can sometimes also say something about the $p$-Sylow subgroup of the class group for $p$ dividing the order of the Galois group of the Galois closure.  So first, we must state a more complete version of the conjecture of \cite{CM90} that includes these primes.

\cut{
CM define a nice theory of random $\fO$-modules, allowing many averages on them to be computed.  Similarly in spirit, our nice theory of moments above is for $\fO$ modules.  But $\Z[\Gamma]\ne\fO$.  Since away from primes dividing $\Gamma$, they are equal, we can make this guess in this case.  When $e$ is a central idempotent, and at $e$ and $S$ they are the same, could also make that guess, need $e$ to make sense first.  Remark: this is not a condition that comes from anything about class groups, but rather one has a nice theory of $\O$-modules and can only use it when desired objects are $\fO$-modules.  
}

In this section we continue the notations introduced in Section~\ref{S:ssNot} and Section~\ref{S: notations for the Heuristics}. In particular, $\Gamma$ is a fixed finite group.

\begin{definition} Let $e$ be any central idempotent of $\Q[\Gamma]$. We say that a prime number $p$ is \emph{good for $e$} if $e\in\Z_{(p)}[\Gamma]$ and $e\Z_{(p)}[\Gamma]$ is a maximal $\Z_{(p)}$-order in $e\Q[\Gamma]$, and it is \emph{bad for $e$} otherwise.
\end{definition}

This definition is stated slightly different from the original one in \cite[6.1]{CM90}, but they are equivalent (see \cite[Theorem 10.5]{R03maximal}). \cut{An order $\fO$ in a semisimple $\Q$-algebra $A$ is maximal if and only if for every irreducible central idempotent $e$ of $A$, the order $\fO$ contains $e$ and $e\fO$ is a maximal order in $eA$, c.f., \cite[Theorem 10.5]{R03maximal}. If $e\in\Z_{(p)}[\Gamma]$ and $e\Z_{(p)}[\Gamma]$ is a maximal order in $e\Q[\Gamma]$, then for any irreducible central idempotent $e'$ contained in $e$, i.e., $e'\cdot e=e'$, it must be contained in $\Z_{(p)}[\Gamma]$, and $e'\Z_{(p)}[\Gamma]=e'\cdot e\Z_{(p)}[\Gamma]$, as the $e'$-component of a maximal order, is a maximal order in $e'\Q[\Gamma]$, hence we obtain the same conditions as in \cite{CM90}.}
A prime $p$ such that $p\nmid\lvert\Gamma\rvert$ is good for any central idempotents $e$, including $e=1$. For a central idempotent $e$ in $\Q[\Gamma]$, and $S$ a set of primes good for $e$, \cite[Hypothesis 6.6]{CM90} conjecture a distribution for $e\Cl^S_K$ (see also Proposition \ref{Prop: M(f)=E(f(X))} and Lemma \ref{lemma: e-relative class group=e-class group}).

\begin{conjecture}[Cohen and Martinet \cite{CM90}]\label{C:CMfull}
Let $e$ be a fixed central idempotent in $(1-e_1)\Q[\Gamma]$, such that $e=e_2+\cdots+e_k$, where the $e_i$ are irreducible central idempotents. Let  $S$ be a set of prime numbers such that if $p\in S$ then $p$ is a good prime for $e$, and $\uu\in \Q^{k-1}$.
Let $X=X(e(1-e_1)\Q[\Gamma] ,\uu,e\Z_S[\Gamma])$.
Then, for every ``reasonable'' non-negative function $f$ defined on the set of isomorphism classes of finite $e\Z_S[\Gamma]$-modules, we have:
\[
\lim_{x\to\infty}\frac{\sum_{|\Disc L|\leq x}f(e\Cl^S_{L})}{\sum_{|\Disc L|\leq x}1}=\mathbb{E}\big(f(X)\big)
\]
where $L$ runs through all \(\Gamma\)-extensions of \(K_0\) such that $\lvert \Disc L\rvert\leq x$ and the rank of $L|K_0$ restricted to the coordinates $2,\dots,k$ is $\uu$.
\end{conjecture}

 For a field extension $L|K$ of number fields with groups of fractional ideals $I_L$ and $I_K$, the embedding $i:I_K\to I_L$ defined on fractional ideals induce, by passing to the classes, the homomorphism:
 \begin{equation*}i_*:\Cl_K\to\Cl_L.\end{equation*}
For this homomorphism, we have the following.
\begin{theorem}[{\cite[Theorem 7.6]{CM90}}] Let $L|K$ be a \(\Gamma'\)-extension of number fields. The kernel (resp. the cokernel) of
 \begin{equation*}i_*:\Cl_K\to\Cl_L^{\Gamma'}\ \text{is annihilated by }\lvert\Gamma'\rvert\text{ (resp. $\lvert\Gamma'\rvert^2$)}.\end{equation*}
\end{theorem}

The direct corollary is the following.
\begin{corollary}[{ \cite[Corollary 7.7]{CM90} }]\label{cor:7.7 in CM90}Let $K_0\subseteq K\subseteq L$ be a tower of extensions such that $L|K_0$ is a $\Gamma$-extension and that $K$ is the fixed field of the subgroup $\Gamma'$ of $\Gamma$. If every prime in $S$ is not a prime divisor of $\lvert\Gamma'\rvert$, the homomorphism
 \begin{equation*}i_*:\Cl^S_{K|K_0}\to\left(\Cl^S_{L|K_0}\right)^{\Gamma'}\ \text{ is an isomorphism.}\end{equation*}
\end{corollary}

When $p\nmid\lvert\Gamma\rvert$, the above results mean that Conjecture~\ref{C:CMfull} implies a distribution on the class group of the fields $K|\Q$ with Galois closure $L|\Q$ (ordered by the discriminant of the Galois closure). 

Now consider the primes $p\mid\lvert\Gamma\rvert$. We'll see below (Lemma \ref{lemma: good primes for augmentation character}) that if $p$ is a good prime for $e_{\Gamma/\Gamma'}$ which is defined below, then $p\mid\lvert\Gamma'\rvert$, which implies that Corollary \ref{cor:7.7 in CM90} is not useful if we want to make predictions on the distribution of $p$-Sylow subgroups of class groups of non-Galois fields for $p\mid |\Gamma|$.
However, in this section we will prove Theorem~\ref{thm: Gamma'-invariant} that allows us  to deduce consequences
 Conjecture~\ref{C:CMfull} for $p$-Sylow subgroups of class groups of non-Galois fields and $p\mid |\Gamma|$.
 
\begin{definition} Let $1_{\Gamma'}$ be the unit character of $\Gamma'$, and
 \[r_{\Gamma/\Gamma'}=\Ind_{\Gamma'}^\Gamma 1_{\Gamma'}\text{ and }a_{\Gamma/\Gamma'}=r_{\Gamma/\Gamma'}-1_\Gamma.\]
Then define $e_{\Gamma/\Gamma'}$ to be the central idempotent associated to $a_{\Gamma/\Gamma'}$, i.e., if $V$ is a representation of $\Gamma$ over $\Q$ with character $a_{\Gamma/\Gamma'}$, then $e_{\Gamma/\Gamma'}$ is the minimal central idempotent of $\Q[\Gamma]$ that acts on $V$ as identity.
\end{definition}

\begin{theorem}\label{thm: Gamma'-invariant} Let $K_0\subseteq K\subseteq L$ be a tower of extensions such that $L|K_0$ is Galois with Galois group $\Gamma$ and that $K$ is the fixed field of the subgroup $\Gamma'$ of $\Gamma$. If every prime $p\in S$ is a good prime for $e_{\Gamma/\Gamma'}$, then 
\begin{enumerate}[(i)]

\item \(p\nmid[K:K_0]\) for all \(p\in S\), and we have the following split short exact sequence
\[
1\longrightarrow\Cl_{K|K_0}^S\longrightarrow\Cl_K^S\stackrel{\Nm}{\longrightarrow}\Cl_{K_0}^S\longrightarrow1,
\]
hence \(\Cl_K^S=\Cl_{K_0}^S\times\Cl_{K|K_0}^S\) where we view \(\Cl_{K_0}^S\) as a subgroup of \(\Cl_K^S\);

\item the induced homomorphism $i_*:\Cl_{K|K_0}^S\to\Cl_L^S$ is injective with image $\big(e_{\Gamma/\Gamma'}\Cl^S_L\big)^{\Gamma'}\subseteq\Cl^S_{L}$, i.e.,
\[
i_*:\Cl_{K|K_0}^S\stackrel{\sim}{\longrightarrow}\big(e_{\Gamma/\Gamma'}\Cl_L^S\big)^{\Gamma'}\text{ is an isomorphism.}
\]
\end{enumerate}\end{theorem}

\begin{remark}
Cohen and Martinet give another result \cite[Theorem 7.8]{CM90} that could be used to relate the class groups of non-Galois fields to Galois fields, but \cite[Theorem 7.8]{CM90} is incorrect as stated.  Their result instead should require that $\Gamma'$ has a normal complement $\Delta$ such that $\Gamma'$ acts on $\Delta$ (by conjugation) with trivial stabilizers on each non-identity orbit.  For example, this hypothesis and the theorem fails for the example $\Gamma=S_4$ and $\Gamma'=S_3$, which is an example that appears in \cite{Cohen1987}.  However, our Theorem~\ref{thm: Gamma'-invariant} can be applied in this case and in every case in which the Cohen-Martinet heuristics make a prediction.
\end{remark}

Note that Theorem 7.4, applied in the case $K_0=\Q$, has the following corollary.
\begin{corollary}\label{C:cap}
Let $L/\Q$ be a \(\Gamma\)-field and \(K\) be the fixed field of \(\Gamma'\). If $p$ is good for $e_{\Gamma/\Gamma'}$, then the order of the capitulation kernel
$$
\ker i_* =\ker( \Cl_K \ra \Cl_L)
$$
is not divisible by $p$.
\end{corollary}

For many pairs $(\Gamma,\Gamma')$, there is at least one prime $p\mid |\Gamma'|$ that is good for $e_{\Gamma/\Gamma'}$, e.g. $p$ is good for $(S_{p+1},S_p)$, 
and $2$ is good for $(A_5,A_4)$, and $5$ is good for $S_5$ or $A_5$ with a certain subgroup of index $6$ (a stabilizer of the action on $\P^1_{\F_5}$).  For these primes, Corollary~\ref{C:cap} appears to be a new result on the capitulation kernel.

From Theorem \ref{thm: Gamma'-invariant}, we see that Conjecture 6.1 implies a conjecture on averages of functions on class groups of non-Galois fields, in which the finite abelian group $H$ appears with weight proportional to
 \begin{equation}\label{eqn: probability of non-Galois fields in section 6}
 \sum_{\begin{subarray}{c}G/\sim\\G^{\Gamma'}\cong H\end{subarray}}\frac{1}{\lvert G\rvert^{\uu}\lvert\Aut_\Gamma(G)\rvert},\end{equation}
where $G$ runs through all finite $e_{\Gamma/\Gamma'}\Z_S[\Gamma]$-modules, up to isomorphism, such that $G^{\Gamma'}\cong H$  as abelian groups.
We'll spend the rest of this section proving Theorem \ref{thm: Gamma'-invariant}. In the next section we will give a simple expression for (\ref{eqn: probability of non-Galois fields in section 6}) and an interpretation of the values appearing in (\ref{eqn: probability of non-Galois fields in section 6}).
We start with a useful statement that we will use repeatedly.

\begin{lemma}\label{lemma: cohomological triviality}
Let $e$ be a central idempotent in $\Q[\Gamma]$ such that $e\in\Z_S[\Gamma]$ and that $e\Z_S[\Gamma]$ is a maximal order in $e\Q[\Gamma]$. Then any $e\Z_S[\Gamma]$-module $G$ is cohomologically trivial as a $\Gamma$-module, i.e., for every subgroup $\Lambda$ of $\Gamma$ and every integer $n\in\Z$, we have
 \[\hat{H}^n(\Lambda,G)=0,\]
where $\hat{H}$ denotes Tate cohomology.
\end{lemma}

\begin{proof}
Note that via the ring homomorphism $e:\Z_S[\Gamma]\to e\Z_S[\Gamma]$ given by $x\mapsto ex$,  all $e\Z_S[\Gamma]$-modules are also $\Gamma$-modules.

Let $G$ be any $e\Z_S[\Gamma]$-module. We can find a projective $e\Z_S[\Gamma]$-module $P$ with surjective homomorphism $\varphi:P\to G$. Then we have a short exact sequence of $e\Z_S[\Gamma]$-modules
 \[0\to L\to P\to G\to0,\]
where $L$ is the kernel of $\varphi$. Since maximal orders are hereditary (e.g., see \cite[Theorem 21.4]{R03maximal}) the submodule $L$ of $P$ is also a projective $e\Z_S[\Gamma]$-module. Since $e\in\Z_S[\Gamma]$, we know that, as $\Gamma$-modules, $e\Z_S[\Gamma]$ is a direct summand of $\Z_S[\Gamma]$. Therefore $P$ and $L$, as summands of the module $(e\Z_S[\Gamma])^m$ for some $m$, are summands of the module $(\Z_S[\Gamma])^m$. Note that $\Z_S[\Gamma]$ is an induced $\Gamma$-module and hence cohomologically trivial. So $P$ and $L$, as summands of some induced $\Gamma$-module, are both cohomologically trivial. Then the short exact sequence implies that $G$ is also cohomologically trivial.
\end{proof}

Next, we note  the following property of the central idempotent $e_{\Gamma/\Gamma'}$ and its relationship to
$$e_1'=\dfrac{1}{\lvert\Gamma'\rvert}\sum_{\tau\in\Gamma'}\tau.$$

\begin{lemma}\label{lem: Frobenius reciprocity} If $V$ is any $\Q$-representation of $\Gamma$ of character $\chi$, then
 \[\dim_\Q V^{\Gamma'}=\langle1_{\Gamma'},\operatorname{Res}^\Gamma_{\Gamma'}\chi\rangle_{\Gamma'}=\langle r_{\Gamma/\Gamma'},\chi\rangle_\Gamma.\]
In particular, let $\chi_1,\dots,\chi_m$ be all the $\Q$-irreducible characters of $\Gamma$ such that $e_i$ is associated to $\chi_i$ for all $i=1,\dots,m$, then
 \[e_ie_1'\neq0\iff e_i=e_1\text{ or }e_i\cdot e_{\Gamma/\Gamma'}=e_i,\,\forall i=1,\dots,m.\]
\end{lemma}

\begin{proof} 
The first identity is exactly given by Frobenius reciprocity.
For the second statement, note that $e_i\Q[\Gamma]$ is a representation of character $n_i\chi_i$ for some $n_i\geq1$, and that
$(e_i\Q[\Gamma])^{\Gamma'}=e_1'e_i\Q[\Gamma].$
\end{proof}

\begin{remark}\label{remark: central idempotent associated to a character} 
We let  $e_1,e_2,\dots,e_k$ be all the distinct irreducible central idempotents of $\Q[\Gamma]$ such that $e\cdot e_1'\neq0$.
By the above lemma, 
 \[e_{\Gamma/\Gamma'}=e_2+\cdots+e_k,\]
which could be taken as an alternative definition for $e_{\Gamma/\Gamma'}$.
\end{remark}

\begin{lemma}\label{lemma: e-relative class group=e-class group} Let \(L|K_0\) be a \(\Gamma\)-extension of number fields. If \(e\) is a central idempotent of \(\Q[\Gamma]\) such that \(e_1\cdot e=0\) and \(p\) is a prime number that is good for \(e\), then
\[
e\Cl_{L}[p^\infty]=e\Cl_{L|K_0}[p^\infty].
\]
\end{lemma}

\begin{remark} This lemma shows that taking the relative class group has no effect if one only cares about good primes for some central idempotent \(e\in\Q[\Gamma]\). Therefore in the statement of the Cohen-Lenstra-Martinet Conjectures (see Conjecture \ref{C:Galois} and \ref{C:CMfull}) we do not need to use the concept of relative class group.\end{remark}

\begin{proof}
First of all let's introduce some notations. For a number field \(k\), let \(I_k\) be the group of fractional ideals and \(P_k\) the group of principal ideals. Then for any prime \(p\), let \(I_{k,p}:=\Z_{(p)}\otimes_\Z I_k\) and \(P_{k,p}:=\Z_{(p)}\otimes_\Z P_{k}\). Note that we have a short exact sequence
\[
1\to P_{k,p}\to I_{k,p}\to\Cl_k[p^\infty]\to1.
\]
Since \(e\in\Z_{(p)}[\Gamma]\), the notion \(e\Cl_L[p^\infty]\) and \(e\Cl_{L|K_0}[p^\infty]\) are well-defined. It is clear that \(e\Cl_{L|K_0}[p^\infty]\subseteq e\Cl_{L}[p^\infty]\). Our goal is to show that \(\Nm_{L|K_0}(I)\) is indeed a principal ideal of \(K_0\) for all ideals \(I\in I_L\) such that the ideal class \([I]\) is contained in \(e\Cl_L[p^\infty]\).

For any \(x\in\Cl_L[p^\infty]\), we have
\[
\Nm_{L|K_0}(ex)=\sum_{\gamma\in\Gamma}\gamma ex=\big(\lvert\Gamma\rvert e_1\big)e\cdot x=0\cdot x=0.
\]
Therefore \(\Nm_{L|K_0}:e\Cl_L[p^\infty]\to e\Cl_L[p^\infty]\) is actually the zero map. Claim: \((eP_{L,p})^\Gamma=P_{K_0,p}\cap eP_{L,p}\). We first prove the claim. Recall that if \(e\Z_{(p)}[\Gamma]\) is a maximal order then any \(e\Z_{(p)}[\Gamma]\)-module is cohomologically trivial by Lemma \ref{lemma: cohomological triviality}. In particular
\[
1=\widehat{H}^0(\Gamma,eP_{L,p})=(eP_{L,p})^\Gamma/\Nm_{L|K}eP_{L,p}
\]
This shows that if a ``principal ideal'' \(I\in eP_{L,p}\) is fixed by \(\Gamma\), then it is represented by a ``principal ideal'' of \(K_0\), hence the claim.

By cohomological triviality again, we know that \(e\Cl_L[p^\infty],eI_{L,p},eP_{L,p}\) are all cohomologically trivial, so
\[
(e\Cl_L[p^\infty])^\Gamma=(eI_{L,p})^\Gamma/(eP_{L,p})^\Gamma=(eI_{L,p})^\Gamma/(P_{K,p}\cap eI_{L,p}).
\]
This implies that for any \(ex\in(e\Cl_{L}[p^\infty])^\Gamma\), we have \(ex=1\) if and only if it is represented by a ``principal ideal'' of \(K\) (an element in \(P_{K,p}\)), hence \(e\Cl_L[p^\infty]\) is indeed generated by ideals whose norm in \(\Cl_{K_0}\) is \(0\), i.e., \(e\Cl_L[p^\infty]=e\Cl_{L|K_0}[p^\infty]\). 
\end{proof}

We need one more lemma for the proof of the theorem.
\begin{lemma}\label{lemma: good primes for augmentation character}
If $p$ is a prime such that $e_{\Gamma/\Gamma'}\in\Z_{(p)}[\Gamma]$, then $p$ does not divide $\lvert\Gamma/\Gamma'\rvert$. In particular, if $p\mid \lvert\Gamma/\Gamma'\rvert$, then $p$ is bad.
\end{lemma}

\begin{proof}
Let
$$P:=\Z_{(p)}[\Gamma]e_1'=\{xe_1'|\,x\in\Z_{(p)}[\Gamma]\}$$
be a left $\Z_{(p)}[\Gamma]$-module. We know that $e_{\Gamma/\Gamma'}e_1'$ is contained in $P$, because $e_{\Gamma/\Gamma'}$ is already contained in $\Z_{(p)}[\Gamma]$. This implies that $e_1=e_1\cdot e_1'$ is also contained in $P$, for the idempotent $e_1'$ is contained in $P$ and could be written as
 \[e_1'=1\cdot e_1'=(e_1+\cdots+e_m)\cdot e_1'=e_1+e_2e_1'+\cdots+e_ke_1'=e_1+e_{\Gamma/\Gamma'}e_1'.\]

Let $\{\sigma_1,\dots,\sigma_q\}$ be a fixed set of representatives of left cosets $\Gamma/\Gamma'$. Then every element $x\in P$ can be written uniquely as
 \[x=\sum_{i=1}^qa_i\sigma_ie_1',\]
where $a_i\in\Z_{(p)}$. If in addition, $x$ is fixed by $\Gamma$, then all the $a_i$ must be the same, which implies that if we let
 \[x_0:=\sum_{i=1}^s1\cdot\sigma_ie_1'=|\Gamma/\Gamma'|\cdot e_1,\]
then $P^\Gamma=\Z_{(p)}x_0$. Since $e_1\in P^\Gamma$, we know that there exists some $a\in\Z_{(p)}$ such that $ax_0=e_1$, i.e.,
 \[a\cdot\lvert\Gamma/\Gamma'\rvert=1.\]
So $\lvert\Gamma/\Gamma'\rvert$ is a unit in $\Z_{(p)}$, i.e., $p$ does not divide $\lvert\Gamma/\Gamma'\rvert$.
\end{proof}

Finally let's prove Theorem \ref{thm: Gamma'-invariant}.

\begin{proof}[Proof of Theorem \ref{thm: Gamma'-invariant}]
It is clear that we can reduce to the case where the set $S$ is the singleton $\{p\}$ with $p$ a good prime for $e_{\Gamma/\Gamma'}$. 

For (i), by Lemma \ref{lemma: good primes for augmentation character}, we know that \(p\nmid\lvert\Gamma/\Gamma'\rvert=[K:K_0]\). Then let's view \(\Cl_{K_0}[p^\infty]\) as a subgroup of \(\Cl_K[p^\infty]\) via the induced map \(i_*:\Cl_{K_0}\to\Cl_K\). We have the following short exact sequence
\[
1\to\Cl_{K|K_0}[p^\infty]\to\Cl_K[p^\infty]\stackrel{n_*}{\to}\Cl_{K_0}[p^\infty]\to1
\]
where \(n_*\) is induced by the norm map \(\Nm_{K|K_0}\), because \(n_*(\Cl_{K_0}[p^\infty])=[K:K_0]\cdot\Cl_{K_0}[p^\infty]=\Cl_{K_0}[p^\infty]\). Then by \(i_*\circ n_*=[K:K_0]\), we see that \(\dfrac{1}{[K:K_0]}i_*\) is well-defined for \(\Cl_{K_0}^S\) and splits \(n_*\). This shows (i).

\hspace{1 mm}

Then let's prove (ii). For a number field $k$, let $I_k$ denote the group of fractional ideals, and $P_k$ the group of principal ideals. Then for $k$, we have the short exact sequence
\[
1\to P_k\to I_k\to\Cl_{k}\to1,
\]
Tensoring with $\Z_{(p)}$ gives us a short exact sequence
\[
1\to\Z_{(p)}\otimes_\Z P_k\to\Z_{(p)}\otimes_\Z I_k\to\Cl_{k}[p^\infty]\to1.
\]
Let $P_{k,p}:=\Z_{(p)}\otimes_\Z P_k$ and $I_{k,p}:=\Z_{(p)}\otimes_\Z I_k$. And for an element $x_k\in I_{k,p}$, we let $[x_k]$ denote its image in the class group.

Recall the set-up in the statement: Let $K_0\subseteq K\subseteq L$ be a tower of extensions such that $\Gal(L|K_0)=\Gamma$ and that $\Gal(L|K)=\Gamma'\subseteq\Gamma$.

\hspace{1 mm}

\item\emph{Claim 1}: By viewing $I_{K,p}$ as a subgroup of $I_{L,p}$ via the embedding $i:I_K\to I_L$, we have an exact sequence
\begin{equation}\label{eqn: Gamma-invariant is principal} I_{K,p}\cap I_{L,p}^\Gamma\to\Cl_K[p^\infty]\to\Cl_{K|K_0}[p^\infty]\to1,\end{equation}
where the map \(\Cl_K[p^\infty]\to\Cl_{K|K_0}[p^\infty]=\Cl_K[p^\infty]/\Cl_{K_0}[^\infty]\) is the quotient map given by (i). Let's prove the claim. First of all \(I_{K_0,p}\subseteq I_{L,p}^\Gamma\), therefore the image of \(I_{K,p}\cap I_{L,p}^\Gamma\) in \(\Cl_K[p^\infty]\) must contain \(\Cl_{K_0}[p^\infty]\). If \(x\in I_{K,p}\cap I_{L,p}^\Gamma\) gives an ideal class \([x]\), then by (i), we can write \([x]=[y]\cdot[z]\) with \([y]\in\Cl_{K_0}[p^\infty]\) and \([z]\in\Cl_{K|K_0}[p^\infty]\). The computation
\[
[x]^{[K:K_0]}=\Nm_{K|K_0}[x]=\Nm_{K|K_0}[y]\cdot\Nm_{K|K_0}[z]=[y]^{[K:K_0]}
\]
shows that \([z]=1\) and \([x]\in\Cl_{K_0}[p^\infty]\). Therefore the image of \(I_{K,p}\cap I_{L,p}^\Gamma\) is exactly \(\Cl_{K_0}[p^\infty]\), the kernel of \(\Cl_K[p^\infty]\to\Cl_{K|K_0}[p^\infty]\).

\hspace{1 mm}

\item\emph{Claim 2}: We have a short exact sequence
 \begin{equation}\label{eqn: 2} 1\to P_{K,p}\cap e_{\Gamma/\Gamma'}P_{L,p}\to I_{K,p}\cap e_{\Gamma/\Gamma'}I_{L,p}\to\Cl_{K|K_0}[p^\infty]\to1.\end{equation}
We prove this claim now. First of all, the ideal classes given by \(I_{K,p}\cap e_{\Gamma/\Gamma'}I_{L,p}\) are contained in the relative class group \(\Cl_{K|K_0}[p^\infty]\), because
\[
\Nm_{K|K_0}[y]=\Nm_{K|K_0}e_{\Gamma/\Gamma'}[y]=\sum_{\sigma\in\Gamma/\Gamma'}\sigma(e_{r_{\Gamma/\Gamma'}}-e_1)\cdot[y]=|\Gamma/\Gamma'|(e_1e_{r_{\Gamma/\Gamma'}}-e_1)\cdot[y]=1.
\]
We then only need to show the surjectivity. As a $\Z_{(p)}[\Gamma]$-module, $I_{L,p}$ admits the following decomposition
 \[I_{L,p}=e_{\Gamma/\Gamma'}I_{L,p}\times(1-e_{\Gamma/\Gamma'})I_{L,p}.\]
Consequently $I^{\Gamma'}_{L,p}=\big(e_{\Gamma/\Gamma'}I_{L,p}\big)^{\Gamma'}\times\big((1-e_{\Gamma/\Gamma'})I_{L,p}\big)^{\Gamma'}$. By $I_{L,p}\hookrightarrow V:=\Q\otimes_{\Z_{(p)}}I_{L,p}$, we know that $x\in I_{L,p}$ is fixed by $\Gamma'$ if and only if $e_1'\cdot x=x$ where the action happens in $V$. Since 
 \[e_1'\cdot(1-e_{\Gamma/\Gamma'})=e_1'\cdot(e_1+e_{k+1}+\cdots+e_m)=e_1'\cdot e_1=e_1,\]
for any element $z\in(1-e_{\Gamma/\Gamma'})V$, it is fixed by $\Gamma'$ if and only if it is fixed by $\Gamma$. Therefore if $x\in I_{L,p}^{\Gamma'}$, then 
\[x=y\cdot z\]
with $y\in\big(e_{\Gamma/\Gamma'}I_{L,p}\big)^{\Gamma'}$ and $z\in I_{L,p}^\Gamma$. By Lemma \ref{lemma: cohomological triviality}, the $e_{\Gamma/\Gamma'}\Z_{(p)}[\Gamma]$-module $e_{\Gamma/\Gamma'}I_{L,p}$ is cohomologically trivial, hence
 \[(e_{\Gamma/\Gamma'}I_{L,p})^{\Gamma'}/\Nm_{L|K}e_{\Gamma/\Gamma'}I_{L,p}=\hat{H}^0(\Gamma',e_{\Gamma/\Gamma'}I_{L,p})=1.\]
Therefore, $y$ is always an element in $I_{K,p}$. If the element above $x=y\cdot z$ is contained in $I_{K,p}$, then $z$ is also contained in $I_{K,p}$, i.e.,
 \[I_{K,p}=\big(I_{K,p}\cap e_{\Gamma/\Gamma'}I_{L,p}\big)\times\big(I_{K,p}\cap I^\Gamma_{L,p}\big),\]
where the direct product is the direct product as abelian groups. Then by (\ref{eqn: Gamma-invariant is principal}), $[z]\in\Cl_{K_0}[p^\infty]$, and $[x]\equiv[y]$ in the relative class group $\Cl_{K|K_0}[p^\infty]=\Cl_K[p^\infty]/\Cl_{K_0}[p^\infty]$, which proves Claim 2. Moreover, the claim also tells us that $i_*(\Cl_{K|K_0}[p^\infty])\subseteq e_{\Gamma/\Gamma'}\Cl_{L|K_0}[p^\infty]$.

\hspace{1 mm}

\item\emph{Final Step}: Since $p$ is a good prime for $e_{\Gamma/\Gamma'}$, we know that $e_{\Gamma/\Gamma'}\in\Z_{(p)}[\Gamma]$ and $e_{\Gamma/\Gamma'}\Z_{(p)}[\Gamma]$ is a maximal order of $e_{\Gamma/\Gamma'}\Q[\Gamma]$, hence obtain the following short exact sequence
 \[1\to e_{\Gamma/\Gamma'}P_{L,p}\to e_{\Gamma/\Gamma'}I_{L,p}\to e_{\Gamma/\Gamma'}\Cl_{L}[p^\infty]\to1,\]
where every object showing up is an $e_{\Gamma/\Gamma'}\Z_{(p)}[\Gamma]$-module. Then by Lemma \ref{lemma: cohomological triviality}, we know that $e_{\Gamma/\Gamma'}P_{L,p}$, $e_{\Gamma/\Gamma'}I_{L,p}$ and $e_{\Gamma/\Gamma'}\Cl_L[p^\infty]$ are all cohomologically trivial as $\Gamma$-modules. So the identity
 \[\big(e_{\Gamma/\Gamma'}\Cl_L[p^\infty]\big)^{\Gamma'}/\Nm_{L|K}e_{\Gamma/\Gamma'}\Cl_L[p^\infty]=\hat{H}^0(\Gamma',e_{\Gamma/\Gamma'}\Cl_L[p^\infty])=1\]
holds. This immediately implies that if $[x]\in (e_{\Gamma/\Gamma'}\Cl_L[p^\infty])^{\Gamma'}$, then $[x]$ is represented by an ideal coming from $K$, and $i_*:\Cl_{K|K_0}[p^\infty]\to (e_{\Gamma/\Gamma'}\Cl_L[p^\infty])^{\Gamma'}$ is surjective. Similarly, by 
 \[\hat{H}^0(\Gamma',e_{\Gamma/\Gamma'}I_{L,p})=1,\text{ and }\hat{H}^0(\Gamma',e_{\Gamma/\Gamma'}P_{L,p})=1 \]
we know that
 \[(e_{\Gamma/\Gamma'}I_{L,p})^{\Gamma'}=I_{K,p}\cap e_{\Gamma/\Gamma'}I_{L,p},\text{ and }(e_{\Gamma/\Gamma'}P_{L,p})^{\Gamma'}=P_{K,p}\cap e_{\Gamma/\Gamma'}P_{L,p}.\]
Also by $\hat{H}^1(\Gamma',e_{\Gamma/\Gamma'}P_{L,p})=1$, we have the short exact sequence
 \[1\to (e_{\Gamma/\Gamma'}P_{L,p})^{\Gamma'}\to (e_{\Gamma/\Gamma'}I_{L,p})^{\Gamma'}\to (e_{\Gamma/\Gamma'}\Cl_{L}[p^\infty])^{\Gamma'}\to1.\]
Then these identities together with the short exact sequence (\ref{eqn: 2}) gives the following commutative diagram which concludes the proof:
 \[\begin{tikzcd}
 &1\arrow{r}&P_{K,p}\cap e_{\Gamma/\Gamma'}P_{L,p}\arrow{r}\arrow[d, equal]&I_{K,p}\cap e_{\Gamma/\Gamma'}I_{L,p}\arrow{r}\arrow[d, equal]&\Cl_{K|K_0}[p^\infty]\arrow{r}\arrow{d}{i_*}&1\\
 &1\arrow{r}&(e_{\Gamma/\Gamma'}P_{L,p})^{\Gamma'}\arrow{r}&(e_{\Gamma/\Gamma'}I_{L,p})^{\Gamma'}\arrow{r}&(e_{\Gamma/\Gamma'}\Cl_{L}[p^\infty])^{\Gamma'}\arrow{r}&1.
 \end{tikzcd}\]
\end{proof}

\section{Reinterpretation of the Cohen-Martinet Heuristics in the non-Galois case}\label{S:Re}

In this section, we reinterpret the distribution on abelian groups from \eqref{eqn: probability of non-Galois fields in section 6} that we have shown are predicted by the Cohen-Martinet heuristics to be the distribution of class groups of non-Galois fields.  Returning to the principle that objects should appears inversely as often as their number of automorphisms, we will see that these class groups of non-Galois fields have certain structure and the distribution is given as inverse to the number of automorphisms of that structure.  We end the sections with several examples for different groups $\Gamma$.

We first define some notation used in this section. Let $\Gamma'$ be a fixed subgroup of $\Gamma$. We've defined the trivial idempotent $e_1$ in Section~\ref{S: notations for the Heuristics}, the augmentation character $a_{\Gamma/\Gamma'}$ and the central idempotent $e_{\Gamma/\Gamma'}$ of $\Q[\Gamma]$ associated to it in Section~\ref{S:moments}.
 Let $e_{r_{\Gamma/\Gamma'}}=e_1+e_{\Gamma/\Gamma'}$ be the central idempotent associated to the character $r_{\Gamma/\Gamma'}$, and $e_1'$ be the irreducible central idempotent associated to the unit character $1_{\Gamma'}$ of $\Gamma'$ in $\Q[\Gamma']$. Note that $e_1'$ is naturally an idempotent in $\Q[\Gamma]$ via the embedding $\Gamma'\hookrightarrow\Gamma$, but it is not necessarily central. Throughout this section, let $S$ be a fixed finite set of \emph{good primes} for $e_{\Gamma/\Gamma'}$ (see definition in Section~\ref{S:NG}), and $\fO\subseteq\Q[\Gamma]$ be a maximal $\Z_S$-order containing the group ring $\Z_S[\Gamma]$. By our assumption, $e_{\Gamma/\Gamma'}\fO$ is exactly $e_{\Gamma/\Gamma'}\Z_S[\Gamma]$.  

\begin{definition} For any $(\Gamma,\Gamma)$-bimodule $M$ and any subgroup $\Lambda$ of $\Gamma$, let ${}^{\Lambda}M$ be the subgroup of $M$ fixed by the action of $\Lambda$ on the left. Similarly $M^{\Lambda}$ is the subgroup fixed by the action of $\Lambda$ on the right. \end{definition}

\textbf{Caution}: The notation $M^{\Lambda}$ is \emph{different} from the use in previous sections, as before we only considered left actions. 
The reason for these two notations is that objects like $\fO$ are $(\Gamma,\Gamma)$-bimodules and we have to distinguish left and right $\Gamma'$-invariant parts.

\subsection{Integral model for the Hecke algebra and Morita equivalence}
First of all, $\Q[\Gamma]$ is a $(\Gamma,\Gamma)$-bimodule, we can consider the subspace ${}^{\Gamma'}\Q[\Gamma]^{\Gamma'}$, which is also called the Hecke algebra, written as $\Q[\Gamma'\backslash \Gamma/\Gamma']$, and which we will write as $e_1'\Q[\Gamma]e_1'$. Note that $e_1'\Q[\Gamma]e_1'$ is a $\Q$-algebra, but its identity $e_1'$ is not the identity of $\Q[\Gamma]$. If $V$ is any left $\Q[\Gamma]$-module, then ${}^{\Gamma'}V$ is naturally a left $e_1'\Q[\Gamma]e_1'$-module. Let $e_1'xe_1'\in e_1'\Q[\Gamma]e_1'$ and $v\in {}^{\Gamma'}V$, then for all $\tau\in\Gamma'$, we have
 \[\tau\cdot(e_1'xe_1'\cdot v)=(\tau e_1'xe_1')\cdot v=e_1'xe_1'\cdot v.\]
This shows that $e_1'xe_1'v$ is still fixed by $\Gamma'$, hence $e_1'xe_1'\cdot {}^{\Gamma'}V\subseteq {}^{\Gamma'}V$. Also for a left $\Q[\Gamma]$-module $V$, we always have
 \[{}^{\Gamma'}V={}^{\Gamma'}(e_{r_{\Gamma/\Gamma'}}V).\]
 
So we see that for $\Q[\Gamma]$-module $V$, the invariants  ${}^{\Gamma'}V$ are naturally a $e_1'\Q[\Gamma]e_1'$-module.   Our goal is  now to construct an integral version of this kind of structure.  Given a finite $\fO$-module $G$, one has a natural action of 
$\mathcal{P}:={}^{\Gamma'}\fO^{\Gamma'}=\fO\cap e_1'\Q[\Gamma]e_1'$ on ${}^{\Gamma'}G$ by reasoning as above.  However, 
in general $\mathcal{P}$ is not even a ring, because if $S$ contains any primes dividing $|\Gamma'|$, then 
$\mathcal{P}$ does not contain a multiplicative identity.  Even if $S$ does not contain any primes dividing $|\Gamma'|$, it is not clear what kind of ring $\mathcal{P}$ is.  We will construct
a ring $\fo$, agreeing with $\mathcal{P}$ when $S$ does not contain primes dividing $|\Gamma|$ and larger than $\mathcal{P}$ otherwise, and show that this larger ring $\fo$ still acts on ${}^{\Gamma'}G$.
After proving several results, in Corollary~\ref{cor: o is a maximal order}, we will see that $\fo$ is actually a maximal order.


\begin{definition} We define
 \[\fo={}^{\Gamma'}(e_{\Gamma/\Gamma'}\fO e_1').\]\end{definition}
 
 We include the factor $e_{\Gamma/\Gamma'}$ because of our intended application to (relative) class groups.  When $\Gamma=S_n$ and $\Gamma'=S_{n-1}$ is the stabilizer of an element, then we have 
 $e_{\Gamma/\Gamma'}\fO =M_{n-1}(\Z_S)$ and $\fo =\Z_S$ (see Example~\ref{Ex:Sn}).  When $\Gamma=D_4$ and $\Gamma'$ is a non-central order $2$ subgroup, we have  $e_{\Gamma/\Gamma'}\fO =\Z_S \times M_{2}(\Z_S)$, and $\fo=\Z_S^2$ (see Example~\ref{Ex:D4}). When $\Gamma=A_5$, we let $\Gamma$ act on $\{1,2,3,4,5\}$ in the usual way and let $\Gamma'$ be the subgroup fixing $1$. Then $e_{\Gamma/\Gamma'}\fO =M_4(\Z_S)$ and $\fo=\Z_S$. As suggested by these examples, we will show in general that 
 $e_{\Gamma/\Gamma'}\fO$ and $\fo$ are Morita equivalent in Theorem~\ref{theorem: equivalence of categories}, even though in general in  they can have more complicated structures as arbitrary maximal orders in sums of matrix algebras over division algebras.  This Morita equivalence will play a central role in our reinterpretation of the prediction of the Cohen-Martinet heuristics in the non-Galois case
 
 We start by showing that $\fo$ is an order of the semisimple $\Q$-algebra $e_1'e_{\Gamma/\Gamma'}\Q[\Gamma]e_1'$.

\begin{proposition}\label{Prop:7.1} Let $e_1,\dots,e_m$ be the distinct irreducible central idempotents of $\Q[\Gamma]$ and $e_{\Gamma/\Gamma'}=e_2+\cdots+e_k$. The $\Q$-algebra $e_1'\Q[\Gamma]e_1'={}^{\Gamma'}\Q[\Gamma]^{\Gamma'}$ is a semisimple $\Q$-algebra whose decomposition into simple components is given by
 \[e_1'\Q[\Gamma]e_1'=\prod_{i=1}^ke_1'e_i\Q[\Gamma]e_1'.\]
The category of $e_1'\Q[\Gamma]e_1'$-modules is equivalent to the category of $e_{r_{\Gamma/\Gamma'}}\Q[\Gamma]$-modules. The subgroup ${}^{\Gamma'}(\fO e_1')$ is a $\Z_S$-order of $e_1'\Q[\Gamma]e_1'$, and $\fo$ is a $\Z_S$-order of $e_1'e_{\Gamma/\Gamma'}\Q[\Gamma]e_1'$.
\end{proposition}

\begin{proof} In the proof, let $A=\Q[\Gamma]$ and $A'=e_1'e_{\Gamma/\Gamma'}\Q[\Gamma]e_1'$.
Note that $e_1'\Q[\Gamma]e_1e_1'=e_1A$ and ${}^{\Gamma'}\big((1-e_{\Gamma/\Gamma'})\fO e_1'\big)={}^\Gamma\fO$, c.f. Lemma \ref{lem: Frobenius reciprocity}. We can focus on $e_{\Gamma/\Gamma'}A$, $A'$ and $\fo$ (the ``nontrivial parts'') in the rest of the proof.

The irreducible central idempotents of $A$ give a decomposition of $A'$
 \[A'=e_2e_1'A'\times\cdots\times e_me_1'A',\]
with each component a $\Q$-algebra because $e_ie_1'$ is central in $A'$. Note that $e_1'\cdot e_i\neq0$ if and only if $e_i=e_1$ or $e_i\cdot e_{\Gamma/\Gamma'}\neq0$ by Lemma \ref{lem: Frobenius reciprocity}. So we have
 \[A'=e_2e_1'A'\times\cdots\times e_ke_1'A'.\]
For any simple $\Q$-algebra $B\cong M_l(D)$ where $D$ is an division algebra and any idempotent $f\in B$, we have $fBf\cong M_{l'}(D)$ for some $l'\leq l$. This can be shown using the decomposition of the identity into mutually orthogonal primitive idempotents by the Krull-Schmidt-Azumaya Theorem, see e.g. \cite[6.12]{RC90methods}.

We apply this result to $e_ie_1'$ for each $i=2,\dots,k$ as follows. The $\Q$-algebra $e_iA$ is simple, and $e_ie_1'$ is an idempotent in $e_iA$. Therefore if $e_iA\cong M_{l_i}(D_i)$ where $D_i$ is some division algebra, then there exists some integer $0<l_i'<l$, such that $e_ie_1'A'=e_ie_1'Ae_ie_1'\cong M_{l_i'}(D_i)$, hence $e_1'e_iA'$ is a simple $\Q$-algebra for all $i=1,\dots,k$. Since $A'$ is the direct sum of finitely many simple $\Q$-algebras, it is a semisimple $\Q$-algebra. 

The equivalence of the category of $e_1'e_iA'$-modules and the category of $e_iA$-modules follows from the fact that they are both matrix algebras over $D_i$, hence $A'$ is Morita equivalent to $e_{\Gamma/\Gamma'}A$. Finally by $e_1'e_1\Q[\Gamma]e_1'=e_1A\cong\Q$, the statements on $e_1'\Q[\Gamma]e_1'$ are all proved.

We now check that $\fo$ is indeed a subring of $A'$. By definition, $\fo$, as the $\Gamma'$-invariant part of an $\Gamma$-module, is an additive abelian group. For all $x,y\in e_{\Gamma/\Gamma'}\fO$ such that $xe'_1,ye_1'\in\fo$, since $\sigma xe_1'=xe_1'$ for all $\sigma\in\Gamma'$, we know that $e_1'xe_1'=xe_1'$, i.e., $xe_1'\in A'$ and $\fo\subseteq A'$ is an additive subgroup. For $xe'_1,ye_1'\in\fo$, we have $xe_1'ye_1'=x(e_1'ye_1')=xye_1'$, which is still an element in $\fo$ because $xy\in e_{\Gamma/\Gamma'}\fO$ and $(xe'_1)ye_1'$ is fixed by $\Gamma'$ on the left.
In particular, $e_1'e_{\Gamma/\Gamma'}$ is contained in $\fo$ and is the identity for $A'$, hence $\fo$ is indeed a subring of $A'$. 

Then let's show that $\fo$ is a $\Z_S$-order in $A'$. We've already showed that $\fo$ is a subring of $A'$. Then we check that $\Q\otimes_{\Z_S}\fo=A'$. Let $x\in e_{\Gamma/\Gamma'}A$, then we can write it as $x=\dfrac{1}{n}y$ with some $n\in\Z$ and $y\in|\Gamma'|^2 e_{\Gamma/\Gamma'}\fO$ because $\Q\otimes e_{\Gamma/\Gamma'}\fO=e_{\Gamma/\Gamma'}A$. Therefore 
 \[e_1'xe_1'=\dfrac{1}{n}\otimes e_1'ye_1'\]
where $e_1'ye_1'\in{}^{\Gamma'}(e_{\Gamma/\Gamma'}\fO)^{\Gamma'}\subseteq\fo$ by our construction.  This shows that $\Q\otimes\fo=A'$. 

Finally we show that $\fo$ is finitely generated as a $\Z_S$-module. Since $e_{\Gamma/\Gamma'}\fO$ is finitely generated as a $\Z_S$-module, say $e_{\Gamma/\Gamma'}\fO=\Z_S\cdot x_1+\cdots+\Z_S\cdot x_N$, then 
 \[\fo\subseteq\fO e_1'=\Z_S\cdot x_1e_1'+\cdots+\Z_S\cdot x_Ne_1',\]
is a submodule of a finitely generated $\Z_S$-module, hence itself finitely generated over $\Z_S$.
\end{proof}

Now we will show that the $\Gamma'$-invariant part of an $e_{\Gamma/\Gamma'}\fO$-module is naturally an $\fo$-module.

\begin{lemma}\label{lemma: Gamma'-invariant part is an o-module} For any finitely generated $e_{\Gamma/\Gamma'}\fO$-module $G$, its $\Gamma'$-invariant part ${}^{\Gamma'}G$ is an $\fo$-module via the action 
 \[(\sigma e_1')\cdot g:=\sigma\cdot g,\]
where the right-hand side is the action of $e_{\Gamma/\Gamma'}\fO$ on $G$, for all $g\in{}^{\Gamma'}G$ and $\sigma e_1'\in\fo$ with $\sigma\in e_{\Gamma/\Gamma'}\fO$.
\end{lemma}

\begin{remark} As the identity of $\fo$, the element $e_{\Gamma/\Gamma'}e_1'$ acts as identity on ${}^{\Gamma'}G$ for any $e_{\Gamma/\Gamma'}\fO$-module $G$ despite the fact that $e_{\Gamma/\Gamma'}e_1'$ is not even contained $e_{\Gamma/\Gamma'}\fO$ in general.
\end{remark}
\begin{remark}
 We can immediately see from Theorem \ref{thm: Gamma'-invariant} that $\Cl^S_{K|K_0}$ is naturally an $\fo$-module. This will be the key to our interpretation of (\ref{eqn: probability of non-Galois fields in section 6}).\end{remark}

\begin{proof} If $\sigma e_1'=\tau e_1'$ with $\sigma,\tau\in e_{\Gamma/\Gamma'}\fO$, then the sum of the coefficients of elements in the same left coset of $\Gamma'$ must be the same, hence $\sigma\cdot g=\tau\cdot g$ for all $g\in{}^{\Gamma'}G$. This shows that the definition does not depend on the choice of $\sigma\in e_{\Gamma/\Gamma'}\fO$. Moreover, since $\sigma e_1'$ is fixed by $\Gamma'$ on the left, we know that $\sigma e_1'g\in{}^{\Gamma'}G$. So we've shown that $\sigma e_1'\cdot g=\sigma g$ gives a well-defined map.

Note that $e_1'\cdot\sigma e_1'=\sigma e_1'$ for all $\sigma e_1'\in\fo$ by definition. If $\sigma_1e_1',\sigma_2e_1'\in\fo$ with $\sigma_1,\sigma_2\in e_{\Gamma/\Gamma'}\fO$, then $\sigma_1e_1'\sigma_2e_1'=\sigma_1\sigma_2e_1'$ which shows that the action is associative. Finally, $\sigma_1e_1'g+\sigma_2e_1'g=(\sigma_1+\sigma_2)g=(\sigma_1+\sigma_2)e_1'g=(\sigma_1e_1'+\sigma_2e_1')g$. So this definition  turns ${}^{\Gamma'}G$ into an $\fo$-module.
\end{proof}

We then prove the equivalence of the category of $e_{\Gamma/\Gamma'}\fO$-modules and the category of $\fo$-modules in the rest of this subsection.

\begin{lemma}\label{lemma:the functor gives Gamma' invariant part when e_1' not contained in O} Given a finitely generated left $e_{\Gamma/\Gamma'}\fO$-module $G$, the left $\fo$-module ${}^{\Gamma'}(e_{\Gamma/\Gamma'}\fO)\otimes_{e_{\Gamma/\Gamma'}\fO}G$ is isomorphic to ${}^{\Gamma'}G$ as $\fo$-modules.
\end{lemma}

\begin{proof} It suffices to prove this for each component of $G$, for $eG$ is a left $\Gamma$-module via the composition $\Z_S[\Gamma]\to\fO\to e\fO$ for each irreducible central idempotent $e$ of $e_{\Gamma/\Gamma'}\Q[\Gamma]$. We then fix $e$ and assume $eG=G$. There is a natural $ e\fO$-isomorphism $\varphi: e\fO\otimes_{e\fO}G\stackrel{\sim}{\to}G$ given by $\sigma\otimes g=\sigma\cdot g$. Note that $\sigma g\in{}^{\Gamma'}G$ for all $\sigma\in{}^{\Gamma'}(e\fO)$. We then obtain an $ee_1'\fo$-morphism $\psi:{}^{\Gamma'}(e\fO)\otimes_{e\fO}G\to{}^{\Gamma'}G$ by restricting $\varphi$ on the subgroup ${}^{\Gamma'}(e\fO)\otimes_{e\fO}G$. Because for all $\tau e_1'\in ee_1'\fo$ where $\tau\in e\fO$ we have
 \[\tau e_1'\varphi(\sigma\otimes g)=\tau e_1'(\sigma g)=\tau\sigma g\quad\quad \varphi\big(\tau e_1'(\sigma\otimes g)\big)=\varphi(\tau\sigma\otimes g)=\tau\sigma g.\]

Claim: $\psi:{}^{\Gamma'}(e\fO)\otimes_{e\fO}G\to{}^{\Gamma'}G$ is an $ee_1'\fo$-isomorphism. 

The morphism $\psi$ is injective because $\varphi$ is. For the proof of surjectivity, we first recall that a morphism $f:H_1\to H_2$ of abelian groups is surjective if and only if $f_p:H_{1,p}\to H_{2,p}$ is surjective for all prime $p$ where $f_p$ and $H_{i,p}$ denote the localization at $p$. In addition, $f_p$ is surjective if and only if $\hat{f}_p:\widehat{H}_{1,p}\to\widehat{H}_{2,p}$ is surjective where $\hat{f}_p$ and $\widehat{H}_{i,p}$ denote the completion at $p$. 

Since $\widehat{\fO}:=\fO\otimes_{\Z_S}\Z_p$ is a maximal $\Z_p$-order in $\Q_p[\Gamma]$ (see \cite[11.6]{R03maximal}) and $\widehat{\fo}:=\fo\otimes_{\Z_S}\Z_p$ is the same as ${}^{\Gamma'}(e_{\Gamma/\Gamma'}\widehat{\fO}e_1')$, the results above go through similarly. In particular, the additive subgroup ${}^{\Gamma'}(e_{\Gamma/\Gamma'}\widehat{\fO})$ of $e_{\Gamma/\Gamma'}\widehat{\fO}$ is a left $\widehat{\fo}$-module by the analogue of Lemma \ref{lemma: Gamma'-invariant part is an o-module}, hence an $(\widehat{\fo},e_{\Gamma/\Gamma'}\widehat{\fO})$-bimodule. So we can reduce the problem to proving 
 \[\widehat{\psi}_p:{}^{\Gamma'}(e\widehat{\fO})\otimes_{e\widehat{\fO}}\widehat{G}\to{}^{\Gamma'}\widehat{G}\]
is surjective for all $p\in S$. 

By abuse of notation, let $\fO$ be a maximal $\Z_p$-order in $\Q_p[\Gamma]$ with $p$ a good prime for $e_{\Gamma/\Gamma'}$, and let $\fo:={}^{\Gamma'}(e_{\Gamma/\Gamma'}\fO e_1')$ just like above. Let $e\Q_p[\Gamma]\cong M_l(D)$ be an isomorphism such that $e\fO\cong M_l(\so)$ where $D$ is a division algebra over $\Q_p$ and $\so\subseteq D$ is the unique maximal $\Z_p$-order in $D$ with the unique two-sided maximal ideal $\fp$, c.f. \cite[(12.8),(17.3)]{R03maximal}. Then the finitely generated $e\fO$-module $G$ admits the following matrix representation
 \[\begin{aligned}G&\cong\begin{pmatrix}\so&\cdots&\so\\
 \so&\cdots&\so\\
 \vdots&\vdots&\vdots\\
 \so&\cdots&\so
 \end{pmatrix}_{l\times m}\oplus\begin{pmatrix}
 \so/\fp^{r_1}&\so/\fp^{r_2}&\cdots&\so/\fp^{r_n}\\
 \so/\fp^{r_1}&\so/\fp^{r_2}&\cdots&\so/\fp^{r_n}\\
 \vdots&\vdots&\vdots&\vdots\\
 \so/\fp^{r_1}&\so/\fp^{r_2}&\cdots&\so/\fp^{r_n}\end{pmatrix}_{l\times n}\\
 &=\big(M_{l\times1}(\so)\big)^m\oplus M_{l\times1}(\so/\fp^{r_1})\oplus\cdots\oplus M_{l\times1}(\so/\fp^{r_n}),\end{aligned}\]
such that the action of $e\fO\cong M_l(\so)$ on $G$ is exactly the left matrix multiplication. We may therefore assume without loss of generality that $G$ is indecomposable, i.e., $G\cong M_{l\times1}(\so)$ if $G$ is projective or $G\cong M_{l\times1}(\so/\fp^r)$ with $r\geq1$ if $G$ is torsion. Let $f$ be the primitive idempotent such that
 \[f\mapsto\begin{pmatrix}1&0&\ldots&0\\
 0&0&\ldots&0\\
 \vdots&\vdots&\vdots&\vdots\\
 0&0&\ldots&0\end{pmatrix}\]
via $e\Q_p[\Gamma]\cong M_l(D)$. There exists a surjective morphism $\pi:e\fO\to G$ given by the composition of $e\fO\to\fO f$ defined by $x\mapsto xf$ and the quotient map $\so\to\so/\fp^r$. Since $e\fO$ is projective, by Lemma \ref{lemma: cohomological triviality}, the induced map ${}^{\Gamma'}(e\fO)\to{}^{\Gamma'}G$ is also surjective. For any $g\in{}^{\Gamma'}G$, there exists $\sigma e\in{}^{\Gamma'}(e\fO)$ such that $\pi(\sigma e)=g$. In particular, by definition of $\pi$, we may assume that $\sigma e=\sigma f\in\fO f$, hence
 \[\psi\big(\sigma f\otimes\pi(f)\big)=\sigma f\cdot\pi(f)=\pi(\sigma f)=g.\]
This proves the surjectivity of $\psi$, hence the lemma.
\end{proof}

\begin{lemma}\label{lemma:Morita equivalence when e_1' not contained in O} If $e$ is a central irreducible idempotent contained in $e_{\Gamma/\Gamma'}$, then the subgroup $e\fO e_1'$ of $e\Q[\Gamma]$ consisting of elements of the form $exe_1'$ with $x\in\fO$ is an $(e\fO,ee_1'\fo)$-bimodule where the right $ee_1'\fo$-action is given by right multiplication in $\Q[\Gamma]$. Then the $(e\fO,e\fO)$-bimodule homomorphism $e\fO e_1'\otimes_{ee_1'\fo}{}^{\Gamma'}(e\fO)\to e\fO$ defined by $exe_1'\otimes y\mapsto exe_1'y$, where the right-hand side is the multiplication in $\Q[\Gamma]$, is surjective.\end{lemma}

\begin{proof} The map is well-defined because $e_1'\cdot y=y$ by multiplication in the group algebra $\Q_p[\Gamma]$, hence the product is actually $exy$ which is contained in $e\fO$.

Just like in the proof of Lemma \ref{lemma:the functor gives Gamma' invariant part when e_1' not contained in O}, we will check the surjectivity locally and use the same abuse of notations for $\fO$ and $\fo$. Let $e\Q_p[\Gamma]\cong M_l(D)$ be an isomorphism of $\Q_p$-algebras with $D$ a division algebra over $\Q_p$ such that $e\fO\cong M_l(\so)$ under the isomorphism where $\so\subseteq D$ is the unique maximal $\Z_p$-order of $D$ with the unique maximal two-sided ideal $\fp$ generated by a prime element $\pi$.

Since $\so$ is given by the valuation on $D$ extended from the valuation on $\Q_p$, there exists a smallest integer $n\in\Z$ such that $e_1'\pi^n\in e\fO$. In particular, there exists at least one unit element in the matrix representation of $e_1'\pi^n$. 

Claim: $e_1'\pi^n$ can generate the whole of $e\fO=M_l(\so)$ as $(e\fO,e\fO)$-bimodule. This can be shown by constructing the usual basis $\{E_{ij}\}$ from $e_1'\pi^n$ via finitely many row/column operations. 

Since $e_1'\pi^n$ is contained in the image, the claim shows that $e\fO e_1'\otimes_{ee_1'\fo}{}^{\Gamma'}(e\fO)\to e\fO$ is surjective, and we prove the lemma.
\end{proof}

We finally have the following.

\begin{theorem}\label{theorem: equivalence of categories} The category of $e_{\Gamma/\Gamma'}\fO$-modules and the category of $\fo$-modules are Morita equivalent via the functors: 
 \[\begin{aligned}
 &{}^{\Gamma'}(e_{\Gamma/\Gamma'}\fO)\otimes_{e_{\Gamma/\Gamma'}\fO}\mathrm{-}:e_{\Gamma/\Gamma'}\fO\mathrm{-Mod}\to\fo\mathrm{-Mod}\\
 &e_{\Gamma/\Gamma'}\fO e_1'\otimes_{\fo}\mathrm{-}:\fo\mathrm{-Mod}\to e_{\Gamma/\Gamma'}\fO\mathrm{-Mod}
 \end{aligned}\]
\end{theorem}

\begin{proof}
Let's denote by $(\,,\,)$ the $e_{\Gamma/\Gamma'}\fO$-balanced bilinear map ${}^{\Gamma'}(e_{\Gamma/\Gamma'}\fO)\times e_{\Gamma/\Gamma'}\fO e_1'\to\fo$ defined by $(x,ye_1')\mapsto xye_1'$. This map is well-defined because $xye_1'\in\fO e_1'$ and $e_1'xye_1'=xye_1'$ is contained in the $\Gamma'$-invariant part.

Similarly let $[\,,\,]$ denote the $\fo$-balanced bilinear map $e_{\Gamma/\Gamma'}\fO e_1'\times{}^{\Gamma'}(e_{\Gamma/\Gamma'}\fO)\to e_{\Gamma/\Gamma'}\fO$ given by $[xe_1',y]\mapsto xe_1'y=xy$. Since these bilinear maps are defined using the multiplication in $\Q_p[\Gamma]$, they satisfy the condition for a Morita context, i.e.,
 \[ze_1'\cdot(x,ye_1')=[ze_1',x]\cdot ye_1',\textrm{ and } z\cdot[xe_1',y]=(z,xe_1')\cdot y.\]
Then $\{e_{\Gamma/\Gamma'}\fO,\fo,{}_{\fo}\big({}^{\Gamma'}(e_{\Gamma/\Gamma'}\fO)\big)_{e_{\Gamma/\Gamma'}\fO},{}_{e_{\Gamma/\Gamma'}\fO}\big(e_{\Gamma/\Gamma'}\fO e_1'\big)_{\fo},(\,,\,),[\,,\,]\}$ forms a Morita context.

The map $e_{\Gamma/\Gamma'}\fO e_1'\otimes{}^{\Gamma'}\fO\to e_{\Gamma/\Gamma'}\fO$ is surjective by Lemma \ref{lemma:Morita equivalence when e_1' not contained in O}. The other map is also surjective because we have 
 \[{}^{\Gamma'}(e_{\Gamma/\Gamma'}\fO)\otimes_{e_{\Gamma/\Gamma'}\fO}e_{\Gamma/\Gamma'}\fO e_1'={}^{\Gamma'}(e_{\Gamma/\Gamma'}\fO e_1')=\fo\]
by Lemma \ref{lemma:the functor gives Gamma' invariant part when e_1' not contained in O}. Then the equivalence and the functors are given by Morita theorem (see \cite[Theorem 3.54]{RC90methods}) directly.
\end{proof}

\begin{corollary}\label{cor: o is a maximal order} The $\Z_S$-order $\fo$ in $e_1'e_{\Gamma/\Gamma'}\Q[\Gamma]e_1'$ is a maximal order.\end{corollary}

\begin{proof} By \cite[11.6]{R03maximal}, it suffices to show that $\widehat{\fo}_p=\fo\otimes_{\Z_S}\Z_p$ is a maximal $\Z_p$-order in $e_1'e_{\Gamma/\Gamma'}\Q_p[\Gamma]e_1'$ for each $p\in S$.

Let $A=\Q_p[\Gamma]$ and $A'=e_1'e_{\Gamma/\Gamma'}\Q[\Gamma]e_1'$. We use the same abuse of notation for $\fO$ and $\fo$ as in the proof of Lemma \ref{lemma:the functor gives Gamma' invariant part when e_1' not contained in O} (i.e., $\fO:=\widehat{\fO}$ and $\fo:=\widehat{\fo}$).

First of all $\fo$ is Morita equivalent to $e_{\Gamma/\Gamma'}\fO$. Since $e_{\Gamma/\Gamma'}\fO$ is hereditary and this property is preserved by Morita equivalence, we know that $\fo$ is also a hereditary ring. Let $e\neq e_1$ be an irreducible central idempotent in $A$ such that $e\cdot e_1'\neq0$, and $eA\cong M_l(D)$ where $D$ is a division algebra over $\Q_p$ and $ee_1'A'\cong M_{l'}(D)$ for some $l'<l$ (see Proposition \ref{Prop:7.1}). By \cite[39.14]{R03maximal}, if $ee_1'\fo$ is a hereditary order in $ee_1'A'$, then it is of the form
 \[ee_1'\fo\cong\begin{pmatrix}
 (\so)&(\fp)&(\fp)&\cdots&(\fp)\\
 (\so)&(\so)&(\fp)&\cdots&(\fp)\\
 (\so)&(\so)&(\so)&\cdots&(\fp)\\
 \vdots&\vdots&\vdots&\vdots&\vdots\\
 (\so)&(\so)&(\so)&\cdots&(\so)
 \end{pmatrix}^{(n_1,\dots,n_r)}\]
where $\so\subseteq D$ is the maximal order in $D$ and $\fp$ its unique maximal ideal and $n_1+\cdots+n_r=l'$ gives the size of the block along the diagonal. 

Assume for contradiction that $ee_1'\fo$ is not maximal. By \cite[17.3]{R03maximal}, we know that $r\geq2$ and there exists at least two non-isomorphic indecomposable projective modules, because a column in the above matrix representation is exactly an indecomposable projective module. But this is already contradiction, for $e\fO$ only admits one indecomposable projective module up to isomorphism. 

Therefore, $ee_1'\fo$ must be of the form $M_{l'}(\so)$, and it is a maximal order of $ee_1'A'$ again by \cite[17.3]{R03maximal}. The argument holds for all $ee_1'$, hence $\fo$ is a maximal order of $A'$.
\end{proof}

\subsection{Random \texorpdfstring{$\fo$}{o}-Module}
From \eqref{eqn: probability of non-Galois fields in section 6}, we were led to wanting to understand the distribution of the abelian groups ${}^{\Gamma'}X$ for our random $e_{\Gamma/\Gamma'}\fO$-modules $X$.  Now, we realize that ${}^{\Gamma'}X$
is naturally an $\fo$-module, so we will instead consider the distribution of $\fo$-modules ${}^{\Gamma'}X$.

One one hand, the random $e_{\Gamma/\Gamma'}\fO$-module $X=X(e_{\Gamma/\Gamma'}\Q[\Gamma],\uu,e_{\Gamma/\Gamma'}\fO)$ defined in Section~\ref{S: notations for the Heuristics}
with $\uu=(u_2,\dots,u_{k})\in\Q^{k-1}$ gives us a random $\fo$-module ${}^{\Gamma'}X$. 
 On the other hand, since $\fo$ is a maximal order in the semisimple $\Q$-algebra $e_1'e_{\Gamma/\Gamma'}\Q[\Gamma]e_1'$, we can also define a random $\fo$-module $Y=(e_1'e_{\Gamma/\Gamma'}\Q[\Gamma]e_1',\uv,\fo)$ with $\uv=(v_2,\dots,v_{k})\in\Q^{k-1}$. We are going to show that for suitably chosen $\uu\in\Q^{k-1}$ and $\uv\in\Q^{k-1}$, the random $\fo$-modules ${}^{\Gamma'}X$ and $Y$ have the same distribution. For simplicity, let 
$$X'= {}^{\Gamma'}X.$$

\begin{theorem}\label{thm: choice of v} 
Let $e_1,\dots,e_m$ be the distinct irreducible central idempotents of $\Q[\Gamma]$ and $e_{\Gamma/\Gamma'}=e_2+\cdots+e_k$. 
Let $\chi_i$ be the $\Q$-irreducible character associated to $e_i$ and $\varphi_i$ be any fixed absolutely irreducible character contained in $\chi_i$ for all $i=2,\dots,k$. Let $X=X(e_{\Gamma/\Gamma'}\Q[\Gamma],\uu,e_{\Gamma/\Gamma'}\fO)$ and $Y=Y(e_1'e_{\Gamma/\Gamma'}\Q[\Gamma]e_1',\uv,\fo)$ with $\uu,\uv\in\Q^{k-1}$ so that $u_i$ corresponds to $e_i$ and $v_i$ corresponds to $e_ie_1'$ for all $i=2,\dots,k$. The random $\fo$-modules $X'={}^{\Gamma'}X$ and $Y$ give the same probability distribution if and only if
 \[v_i=\frac{\langle\varphi_i,a_{\Gamma}\rangle}{\langle\varphi_i,a_{\Gamma/\Gamma'}\rangle}u_i\]
for all $i=2,\dots,k$, where $a_\Gamma=a_{\Gamma/1}:=-1+\Ind_1^\Gamma 1$ is the augmentation character of the trivial subgroup.
\end{theorem}

\begin{proof} We will start by obtaining the formula for the probability distribution of $X'$. For any finite $\fo$-module $H$, we have $X'\cong H$ if and only if $X\cong e_{\Gamma/\Gamma'}\fO e_1'\otimes_{\fo}H$ by Theorem~\ref{theorem: equivalence of categories}. Therefore for any two finite $\fo$-modules $H_1,H_2$, let $G_i:=e_{\Gamma/\Gamma'}\fO e_1'\otimes_{\fo}H_i$ for $i=1,2$, and we have
 \[\frac{\mathbb{P}(X'\cong H_1)}{\mathbb{P}(X'\cong H_2)}=\frac{\lvert G_2\rvert^{\uu}\lvert\Aut_{e_{\Gamma/\Gamma'}\fO}(G_2)\rvert}{\lvert G_1\rvert^{\uu}\lvert\Aut_{e_{\Gamma/\Gamma'}\fO}(G_1)\rvert}=\frac{\lvert G_2\rvert^{\uu}\lvert\Aut_\fo(H_2)\rvert}{\lvert G_1\rvert^{\uu}\lvert\Aut_\fo(H_1)\rvert}.\]

Given any finite $\fo$-module $H$, let $G:=e_{\Gamma/\Gamma'}\fO e_1'\otimes H$ be the finite $e_{\Gamma/\Gamma'}\fO$-module such that ${}^{\Gamma'}G\cong H$. By \cite[Theorem 7.3]{CM90}, for each $i=2,\ldots,k$, there exists some finite $\Z_S$-module $G_i$ such that
 \begin{equation}\label{eqn:7.3 in CM90} e_iG\cong G_i^{\langle\chi_i,a_{\Gamma}\rangle}\quad\text{ and }\quad e_ie_1'H={}^{\Gamma'}(e_iG)\cong G_i^{\langle\chi_i,a_{\Gamma/\Gamma'}\rangle},\end{equation}
where the isomorphisms are isomorphisms as abelian groups. We then know that 
\begin{equation}\label{E:utov}
|e_iG|=|e_ie_1'H|^{\langle\chi_i,a_{\Gamma}\rangle/\langle\chi_i,a_{\Gamma/\Gamma'}\rangle}.
\end{equation}
Therefore if
 \[v_i=\frac{\langle\varphi_i,a_{\Gamma}\rangle}{\langle\varphi_i,a_{\Gamma/\Gamma'}\rangle}u_i\]
for all $i=2,\dots,k$, then $\lvert G\rvert^{\uu}=\lvert H\rvert^{\uv}$, hence $X'$ is defined the same way as $Y$ and they give the same probability distribution.

Conversely if $X'$ and $Y$ give the same distribution, then
 \[\frac{\lvert G_2\rvert^{\uu}}{\lvert G_1\rvert^{\uu}}=\frac{\lvert H_2\rvert^{\uv}}{\lvert H_1\rvert^{\uv}}\]
for all finite $e_{\Gamma/\Gamma'}\fO$-modules $G_1,G_2$ such that $H_i:={}^{\Gamma'}G_i$ with $i=1,2$. Then the identities (\ref{eqn:7.3 in CM90}) tell us that this condition forces
 \[v_i=\frac{\langle\varphi_i,a_{\Gamma}\rangle}{\langle\varphi_i,a_{\Gamma/\Gamma'}\rangle}u_i\]
for all $i=2,\dots,k$. 
\end{proof}

\begin{definition} Let $L|K_0$ be a \(\Gamma\)-extension and $\uu\in\Q^m$ be the rank of $L|K_0$. Then define $\uv\in\Q^{k-1}$ given by the formula in Theorem \ref{thm: choice of v} to be the rank of $L^{\Gamma'}|K_0$.
(In Section~\ref{S:ind} we show this does not depend on $L$, but only $L^{\Gamma'}$.)
\end{definition}

Just like in Section~\ref{S:u}, we can express  $\lvert H\rvert^{\uv}$ in terms of the decomposition groups $\Gamma_v$ at infinite places $v|\infty$.

\begin{corollary}\label{C:uandv}
 If $\uu$ is given by the rank of a \(\Gamma\)-extension $L|K_0$ and $\uv$ the rank of $L^{\Gamma'}|K_0$ (as given in the definition just above), then for any finite $\fo$-module $H$, we have
 \[\lvert H\rvert^{\uv}=\lvert e_{\Gamma/\Gamma'}\fO e_1'\otimes_\fo H\rvert^{\uu}=\prod_{v|\infty}\lvert(e_{\Gamma/\Gamma'}\fO e_1'\otimes_\fo H)^{\Gamma_v}\rvert\]
where $v$ runs over all infinite places of $K_0$.
\end{corollary}

\begin{proof} This is the combination of Theorem \ref{thm: choice of v} and Theorem \ref{thm:u in Cohen-Martinet}.
\end{proof}

By  Theorem~\ref{thm: choice of v}, we can always identify the random \(\fo\)-module ${}^{\Gamma'}X$ with some random $\fo$-module $Y=Y(e_1'e_{\Gamma/\Gamma'}\Q[\Gamma]e_1',\uv,\fo)$ and the Cohen-Martinet conjecture predicts $\Cl_K^S$ are distributed as random $\fo$-modules.

\begin{theorem}\label{thm: punchline} Let $\Gamma$ be a finite group and $\Gamma'\subseteq\Gamma$ a subgroup. Assume that $S$ only contains good primes for $e_{\Gamma/\Gamma'}$. If $\uu$ is the rank of some \(\Gamma\)-extension $L|K_0$, then let $\uv$ be the rank of $L^{\Gamma'}|K_0$ (as given in the definition just above) and $Y=Y(e_1'e_{\Gamma/\Gamma'}\Q[\Gamma]e_1',\uv,\fo)$ be the random finite $\fo$-module. For a non-negative function $f$ defined on the class of isomorphism classes of finite $\fo$-modules, the Cohen-Martinet conjecture (Conjecture~\ref{C:CMfull} for $f({}^{\Gamma'}-)$ and $e=e_{\Gamma/\Gamma'}$) implies that
 \[\lim_{x\to\infty}\frac{\sum_{|d_L|\leq x}f(\Cl^S_{L^{\Gamma'}/K_0})}{\sum_{|d_L|\leq x}1}=\mathbb{E}\big(f(Y)\big),\]
where the sums are over \(\Gamma\)-extensions $L|K_0$ and the discriminant $\lvert d_L\rvert\leq x$ and the rank of $L|K_0$ is $\uu$.
\end{theorem}


In particular, the results of Section~\ref{S:moments} all apply here to give the moments of the predicted distributions and see that the distributions are determined by their moments.  

\begin{remark}
The probabilities in Theorem~\ref{thm: punchline} are
$$
\frac{c}{|H|^{\uv}|\Aut_{\fo}(H)|}
$$
for each finite $\fo$-module $H$.  
We also see that if we want the probability of obtaining some finite abelian group $H$, then the desired probability in \eqref{eqn: probability of non-Galois fields in section 6} can be rewritten as a sum over $\fo$-module structures on the finite abelian group $H$ of the above probabilities.
One could  also apply the class triples approach of Section~\ref{S:getaut} to obtain probabilities that are purely inversely proportional to automorphisms of some object.  Perhaps the simplest way to do this to make a class triple from $e_{\Gamma/\Gamma'}\Cl_L^S$.
\end{remark}

\subsection{Examples}
In this section, we give some examples of specific $\Gamma$ and $\Gamma'$ to see what $\fo$ is in that case.
Given a finite group $\Gamma$ and subgroup $\Gamma'$, we define $e_i,\chi_i,\varphi_i$ as in Theorem~\ref{thm: choice of v}.
We have that $e_i\Q[\Gamma]\isom M_{l_i}(D_i),$ where $D_i$ is a division algebra with center $K_i$, and $K_i$ is the field generated by the values of $\varphi_i$.
We can decompose
$$
a_{\Gamma/\Gamma'}=\sum_{i=2}^{k} a_i\chi_i.
$$
for positive integers $a_i.$  Then we can see from the proof of Proposition~\ref{Prop:7.1} and a dimension calculation using Frobenius reciprocity that
$$
e_1' e_{\Gamma/\Gamma'} \Q[\Gamma]e_1' \isom \bigoplus_{i=2}^k M_{a_i}(D_i).
$$
From this we conclude the following about the cases in which there is really no additional structure by realizing the class group is an $\fo$-module.
\begin{proposition}\label{Prop: criterion} The maximal $\Z_S$-order $\fo$ in $e_1'e_{\Gamma/\Gamma'}\Q[\Gamma]e_1'$ is isomorphic to $\Z_S$ if and only if $a_{\Gamma/\Gamma'}$ is absolutely irreducible.\end{proposition}

\begin{example}[$a_{\Gamma/\Gamma'}$ multiplicity $1$]
So if all the $a_i$ are $1$ and $D_i=K_i$ (i.e. all the Schur indices are 1), or equivalently, every absolutely irreducible character that appears in 
$a_{\Gamma/\Gamma'}$ appears with multiplicity $1$, then by Corollary~\ref{cor: o is a maximal order}, we have that 
$$
\fo \isom \bigoplus_{i=2}^k \Z_{K_i},
$$
where $\Z_{K_i}$ is the localization of the ring of algebraic integers of $K_i$ at by the non-zero rational integers not in $S.$

If in addition, all the decomposition groups $\Gamma_v$ are trivial for a Galois $\Gamma$-extension $L/K_0$, then for the associated $v_i$ for $L^{\Gamma'}$, we can compute using Theorem~\ref{thm: choice of v} that
$v_i=r_K l_i,$ where $r_K$ is the number of infinite places of $K$.
\end{example}

\begin{example}[An example on $S_n$] \label{Ex:Sn}
Even more specifically, we consider the case where $K|\Q$ is a non-Galois extension whose Galois closure $L|\Q$ is a $\Gamma=S_n$-field such that $K$ is the fixed field of $\Gamma'=S_{n-1}$ where $S_{n-1}\hookrightarrow S_n$ in the usual way. Moreover assume that $L|\Q$ is totally real, so $\underline{u}=\underline{1}$ by Theorem~\ref{thm:u in Cohen-Martinet}. Since $a_{\Gamma/\Gamma'}$ is absolutely irreducible with $a_{\Gamma/\Gamma'}\big(23\cdots(n-1)\big)=1$, we have
 \[a_{\Gamma/\Gamma'}=\frac{a_{\Gamma/\Gamma'}(1)}{\lvert\Gamma\rvert}\big((23\cdots(n-1))^{-1}+\cdots\big)=\frac{n-1}{n!}\big((23\cdots(n-1))^{-1}+\cdots\big).\]
 Also, for $p\nmid n!/(n-1)$, one can explicitly compute $e_{\Gamma/\Gamma'}\Z_{(p)}[\Gamma]=
M_{n-1}(\Z_{(p)}).$ 
Therefore $p$ is a good prime if and only if $p\nmid n!/(n-1)$. Let $S$ be the set of good primes for $e_{\Gamma/\Gamma'}$.
 
By Theorem~\ref{thm: choice of v} 
we  have
 \[\lvert H\rvert^v=\lvert H\rvert^{n-1}\]
where $n=\lvert\Gamma/\Gamma'\rvert$, i.e., $v=n-1$.
In this case, $\fo$ is just $\Z_S$.
 Hence we expect $\Cl_K^S$ to behave like a random abelian  group without any additional structure coming from the $\fo$ action, and the predictions have each finite abelian $\Z_S$-module $H$ appearing with probability $|H|^{-(n-1)}|\Aut(H)|^{-1}$ as $\Cl_K^S$.
\end{example}

\begin{example}[An example on $D_4$] \label{Ex:D4}
Let $\Gamma=D_4$, the dihedral group of order $8$ and $S$ only contain odd primes. Write $\Gamma=\langle\sigma,\tau\rangle$ with $\tau^2=\sigma^4=1$ and $\tau\sigma\tau^{-1}=\sigma^{-1}$. Let $K|\Q$ be a degree $4$ extension with Galois closure $L|\Q$ a totally real \(\Gamma\)-field such that $K$ is the fixed field of the subgroup $\Gamma'=\{1,\tau\}$ (so $\underline{u}=\underline{1}$ by Theorem~\ref{thm:u in Cohen-Martinet}). 

The character $a_{\Gamma/\Gamma'}$ is of degree $3$, the sum of two absolutely irreducible characters $\varphi$ of degree $1$, and $\chi$ of degree $2$. Let $e_\varphi$, resp. $e_\chi$, be the irreducible central idempotent in $\Q[\Gamma]$ associated to $\varphi$, resp. $\chi$. The idempotents are given by
$$e_\chi=\frac{1}{8}(1+\sigma^2-\sigma-\sigma^3+\tau+\sigma^2\tau-\sigma\tau-\sigma^3\tau)\quad\text{and}\quad e_\varphi=\frac{1}{2}(1-\sigma^2)$$
and $2$ is the only bad prime number for $e_{\Gamma/\Gamma'}$.

Since $\varphi$ is an absolutely irreducible character of degree $1$ and $e_1'\cdot e_\varphi=e_\varphi$, we then know that $e_\varphi e_1'\fo\cong\Z_S$. On the other hand, Frobenius reciprocity shows that $\dim_\Q e_1'e_\chi\Q[\Gamma]e_1'=1$, hence $e_\chi e_1'\fo$, as a maximal order in $e_1'e_\chi\Q[\Gamma]e_1'$, is also isomorphic to $\Z_S$. So $\fo =\Z_S^2$ as an algebra. 

On the other hand, the normalizer of $\Gamma'$ is $\{1,\tau,\sigma^2,\sigma^2\tau\}$, i.e., there exists $2$ automorphisms of $K|\Q$. In particular, the class group $\Cl_K^S$ is not only an abelian group but an abelian group with an order $2$ automorphism, i.e., $\Cl_K^S$ is a $\Z_S[t]/(t^2-1)$-module with $t\cdot x=\sigma^2\cdot x$. Moreover, one can check that the ring homomorphism $e_{\Gamma/\Gamma'}e_1'\fo\to\Z_S[t]/(t^2-1)$ given by 
$$e_\varphi e_1'\mapsto\dfrac{1}{2}(1+t)\quad\text{and}\quad e_\chi e_1'\mapsto\dfrac{1}{2}(1-t)$$
is an isomorphism which is compatible with the actions on class groups. So in this example, considering the $\fo$-module structure on $\Cl_K^S$ and the structure on $\Cl_K^S$ from the automorphisms of $K|\Q$ are equivalent.

We will also work out the predicted moments explicitly in this case.
Let $X=(e_{\Gamma/\Gamma'}\Q[\Gamma],\underline{1},e_{\Gamma/\Gamma'}\fO)$, and let $G$ be a finite $e_{\Gamma/\Gamma'}\fO$-module, and $H={}^{\Gamma'}G$. Then
 \[\mathbb{E}\left(\lvert\Sur_\fo({}^{\Gamma'}X,H)\rvert\right)=\mathbb{E}\left(\lvert\Sur_{e_{\Gamma/\Gamma'}\fO}(X,G)\rvert\right)=\frac{1}{|G|^{\uu}}=\frac{1}{|e_\varphi G|}\frac{1}{|e_\chi G|}.\]
Then using \eqref{E:utov}, we have
 \[\mathbb{E}\left(\lvert\Sur_\fo({}^{\Gamma'}X,H)\rvert\right)=\frac{1}{\lvert \frac{1+t}{2} H\rvert\lvert \frac{1-t}{2} H\rvert^2}.\]
\end{example}

\begin{example}[An Example on $A_5$]\label{Ex:A5}
This is an example where the non-Galois extension admits no ``automorphism'' but the ring $\fo$ is nontrivial.

Let $\Gamma=A_5$.  The subgroup $\Gamma'$ generated by $(123)$ and $(12)(45)$ is called the twisted $S_3$ in $A_5$ because this subgroup is isomorphic to $S_3$. It is a maximal proper subgroup of $A_5$. Since $\Gamma$ is simple, this says that the normalizer of $\Gamma'$ is itself. 

Now assume that $K|\Q$ is a non-Galois extension with Galois closure a \(\Gamma\)-field $L|\Q$ such that $K=L^{\Gamma'}$. Since automorphisms of $K$ over $\Q$ correspond to $\Gamma'$ cosets of elements $\sigma\in\Gamma$ such that $\sigma\Gamma'\sigma^{-1}=\Gamma'$, then we can see that $K$ admits no nontrivial automorphism. 

The character $r_{\Gamma/\Gamma'}$ is given by a $\Q$-representation of dimension $10$. By checking the character table, $\Gamma$ has $4$ characters over $\Q$. Note that there is a unit character contained in $r_{\Gamma/\Gamma'}$. The character $r_{\Gamma/\Gamma'}$ contains three different absolutely irreducible characters, say $r_{\Gamma/\Gamma'}=\chi_1+\chi_2+\chi_3$ where $\chi_1$ is the unit character, $\chi_2$ is the character of degree $4$ and $\chi_3$ is the character of degree $5$. By Theorem \ref{theorem: equivalence of categories}, this implies that $\fo$ admits two orthogonal irreducible idempotents, hence cannot be isomorphic to $\Z_S$. By computations using Frobenius reciprocity, we can see that $e_1'e_i\Q[\Gamma]e_1'$ is a one-dimensional $\Q$-vector space where $e_i$ is the irreducible central idempotent associated to $\chi_i$, for $i=2,3$. Therefore the ring $\fo$ is isomorphic to $\Z^2_S$. Moreover,
we can check that a prime number $p$ is good for $e_{\Gamma/\Gamma'}$ if and only if $p\neq2,3,5$, i.e., $p\nmid\lvert\Gamma\rvert$. So for a set $S$ of good primes, the class group $\Cl_K^S$ has a natural order $2$ automorphism (from $(1,-1)\in\Z_S^2$) and the conjectures reflect this structure.
\end{example}

\section{Independence of Galois field}\label{S:ind}

Though we imagine the reader was thinking of $L$ as the Galois closure of $K$ in the last two sections, that was never strictly required.  It could have also been a larger Galois extension.  In fact, we could have even considered $\Gamma'$ normal so that $K/K_0$ was Galois.  With this realization, we see that the Cohen-Martinet heuristics make several (infinitely many) predictions for the averages of the the same class groups (though each prediction is with a different ordering of the fields, since the conjectures as worded are always ordered by the discriminants of the Galois fields).  In this section, we show that all those predictions agree.

We starting by showing that  $\uv$ does not depend on the choice of the Galois extension $L|K_0$ containing $K|K_0$ (see the explicit statement below). We start with a lemma, whose proof is straightforward.

\begin{lemma}\label{L:twogroups} If $\Gamma'\subseteq\Gamma$ is a normal subgroup, then $e_1'$ is central in $\Q[\Gamma]$ and
 \[(e_1+e_{\Gamma/\Gamma'})\Q[\Gamma]\cong e_1'\Q[\Gamma]\cong\Q[\Gamma/\Gamma'].\]
In particular, if we let $\bar{e}_1$  be the irreducible central idempotent in $\Q[\Gamma/\Gamma']$ associated to the unit character of $\Gamma/\Gamma'$, then the maximal order $\fo$ of $e_1'e_{\Gamma/\Gamma'}\Q[\Gamma]$ is isomorphic to a maximal order in $(1-\bar{e}_1)\Q[\Gamma/\Gamma']$.
\end{lemma}

\cut{
\begin{proof}
Since $\Gamma'$ is normal, we then know that if $\{\sigma_1,\dots,\sigma_n\}$ is a fixed set of representatives of $\Gamma/\Gamma'$, then the morphism $e_1'\Q[\Gamma]\to\Q[\Gamma/\Gamma']$ given by
 \[\sum_{i=1}^na_{\sigma_i}\sigma_ie_1'\mapsto\sum_{i=1}^na_{\sigma_i}\sigma_i\Gamma'\]
is an isomorphism of $\Q$-algebras. 

On the other hand, we'll show that $e_1'$ is exactly $e_1+e_{\Gamma/\Gamma'}$. The central idempotent $e_1+e_{\Gamma/\Gamma'}$ is associated to the character $r_{\Gamma/\Gamma'}$, which is the character of $\bigoplus_{i=1}^n\sigma_iV$ where $V$ is the trivial representation of $\Gamma'$. By $\sigma\sigma_iV=\sigma_iV$ if and only if $\sigma\in\Gamma'$ for all $i=1,\dots,n$, we know that 
 $$r_{\Gamma/\Gamma'}(\sigma)=\left\{\begin{aligned}
 &n=\lvert\Gamma/\Gamma'\rvert\quad&\text{if }\sigma\in\Gamma'\\
 &0\quad&\text{otherwise.}
 \end{aligned}\right.$$
Since
 \[\frac{\lvert\Gamma/\Gamma'\rvert}{\lvert\Gamma\rvert}\sum_{\sigma\in\Gamma}r_{\Gamma/\Gamma'}(\sigma)=\frac{1}{\lvert\Gamma'\rvert}\sum_{\tau\in\Gamma'}\tau=e_1'\]
already gives us a central idempotent, this implies that $e_1+e_{\Gamma/\Gamma'}=e_1'$. Then the first isomorphism $(e_1+e_{\Gamma/\Gamma'})\Q[\Gamma]\cong e_1'\Q[\Gamma]$ is clear.

The statement on maximal orders is just a direct consequence of the isomorphism of the $\Q$-algebras.
\end{proof}}

\begin{theorem} Let $K|K_0$ be any finite extension with Galois closure a \(\Gamma\)-extension $L|K_0$ of rank $\uu\in\Q^{m-1}$ such that \(\Gal(L|K)\cong\Gamma\). Let $M|K_0$ be a \(\Sigma\)-extension of rank $\underline{w}\in\Q^{n-1}$ such that $L\subseteq M$ with $\Gal(M|L)\cong\Delta$, and $\Gal(M|K)\cong\Sigma'$. If $S$ only contains good primes for $e_{\Gamma/\Gamma'}\in\Q[\Gamma]$ and $e_{\Sigma/\Sigma'}\in\Q[\Sigma]$, then the rank $\uv$ of $L^{\Gamma'}|K_0$ and the rank $\tilde{\uv}$ of $M^{\Sigma'}|K_0$ are the same. Moreover, ${}^{\Gamma'}(e_{\Gamma/\Gamma'}\fO)$ is isomorphic to ${}^{\Sigma'}(e_{\Sigma/\Sigma'}\widetilde{\fO})$ where $\fO$, resp. $\widetilde{\fO}$, is a maximal $\Z_S$-order in $\Q[\Gamma]$, resp. in $\Q[\Sigma]$ provided that the embedding $\Q[\Gamma]\to\Q[\Sigma]$ defined by 
 $$\gamma\mapsto\sigma\sum_{\delta\in\Delta}\delta$$
where $\gamma$ is the image of $\delta$ under the surjective map $\Sigma\to\Gamma$, maps $\fO$ into $\widetilde{\fO}$.
\end{theorem}

\begin{proof}
We use $E$ for central idempotents in $\Q[\Sigma]$ and $e$ for the ones in $\Q[\Gamma]$. For example let $e_1':=\dfrac{1}{\lvert\Gamma'\rvert}\sum_{\gamma\in\Gamma'}\gamma$, $E_1':=\dfrac{1}{\lvert\Sigma'\rvert}\sum_{\sigma\in\Sigma'}\sigma$. Moreover let $F_1:=\dfrac{1}{\lvert\Delta\rvert}\sum_{\delta\in\Delta}\delta$.
Note that $E_1'\cdot F_1=F_1\cdot E_1'=E_1'$. By Lemma~\ref{L:twogroups}, we have
 \[\begin{aligned}
 E_1'\Q[\Sigma]E_1'&=E_1'F_1\Q[\Sigma]E_1'=\dfrac{1}{\lvert\Sigma'\rvert}\sum_{\sigma\in\Sigma'}\sigma\cdot\Q[\Sigma/\Delta]E_1'\\
 &=\frac{1}{\lvert\Sigma'/\Delta\rvert}\sum_{\sigma\Delta\in\Sigma'/\Delta}\sigma\Delta\cdot\Q[\Sigma/\Delta]E_1'\cong e_1'\Q[\Gamma]e_1'
 \end{aligned}\]
This computation shows that ${}^{\Gamma'}(e_{\Gamma/\Gamma'}\fO)$ is equivalent to ${}^{\Sigma'}(e_{\Sigma/\Sigma'}\widetilde{\fO})$, because they are both maximal orders in $e_1'\Q[\Gamma]e_1'$. Moreover, if the embedding $\Q[\Gamma]\hookrightarrow\Q[\Sigma]$ sends $\fO$ into $\widetilde{\fO}$, then by the isomorphism in Lemma~\ref{L:twogroups} $\Q[\Gamma]\cong(E_1+E_{\Sigma/\Delta})\Q[\Sigma]$ which is induced by the embedding, we know that $\fO\cong(E_1+E_{\Sigma/\Delta})\widetilde{\fO}$, hence the isomorphism ${}^{\Gamma'}(e_{\Gamma/\Gamma'}\fO)\cong{}^{\Sigma'}(e_{\Sigma/\Sigma'}\widetilde{\fO})$.

Note that by Lemma~\ref{L:twogroups}, $E_{\Sigma/\Delta}\Q[\Sigma]\cong(1-e_1)\Q[\Gamma]$, and they have same number of irreducible components whose correspondence is given by $E\mapsto EF_1$ for all irreducible central idempotents $E\in\Q[\Sigma]$. Assume without loss of generality that $E_{\Sigma/\Delta}=E_2+\cdots+E_m$ in $\Q[\Sigma]$ and $E_iF_1=e_i$ for all $i=2,\dots,m$.

Claim: $w_i=u_i$ for all $i=2,\dots,m$. Let's prove the claim. Let $\Sigma_v\subseteq\Sigma$ be any decomposition group of some infinite place $v|\infty$ of $K_0$ defined up to conjugacy. Note that the ranks $\uu$ and $\underline{w}$ do not depend on the choice of the maximal orders. We may assume without loss of generality that $E_{\Sigma/\Delta}\widetilde{\fO}\cong(1-e_1)\fO$. Let $G$ be any finite $E_{\Sigma/\Delta}\widetilde{\fO}$-module and $H$ the corresponding $(1-e_1)\fO$-module under the isomorphism of maximal orders. Since $G$ is fixed by $\Delta$, hence a $\Sigma/\Delta$-module, and we can take $\Gamma_v$ to be the image of $\Sigma_v$ under the surjective map $\Sigma\to\Gamma$, and obtain
 $$\lvert{}^{\Sigma_v}G\rvert=\lvert{}^{\Sigma_v\Delta}G\rvert=\lvert{}^{\Gamma_v}H\rvert.$$
By Theorem \ref{thm:u in Cohen-Martinet}, we know that the claim is true. 

Then by the interpretation of $\uv$ for non-Galois case and the fact that we can choose the maximal orders such that $E_{\Sigma/\Delta}\widetilde{\fO}\cong(1-e_1)\fO$, we know that the computation of the rank $\uv$ of $K|K_0$ can always be reduced to its Galois closure $L|K_0$, i.e., the rank $\uv$ of $K|K_0$ is a property of $K$ and the distribution of the random $\fo$-module $Y=(e_1'e_{\Gamma/\Gamma'}\Q[\Gamma]e_1',\uv,\fo)$ does not depend on the choice of the Galois extension $M|K_0$ containing $K$.
\end{proof}

\subsection*{Acknowledgements} 
The authors would like to thank Joseph Gunther, Yuan Liu, and Jiuya Wang  for useful conversations related to this work.  We would like to thank Alex Bartel for numerous comments that helped improve the exposition of this paper.
This work was done with the support of an American Institute of Mathematics Five-Year Fellowship, a Packard Fellowship for Science and Engineering, a Sloan Research Fellowship, and National Science Foundation grants DMS-1301690 and DMS-1652116.

\newcommand{\etalchar}[1]{$^{#1}$}

\end{document}